\documentclass[a4paper,11pt]{amsart}

\usepackage[T1]{fontenc}
\usepackage[sc]{mathpazo}
\RequirePackage{calrsfs}
\DeclareSymbolFont{rsfscript}{OMS}{rsfs}{m}{b}
\DeclareSymbolFontAlphabet{\mathrsfs}{rsfscript}

\def\bfit{\bfseries\itshape}

\input xypic
\xyoption{all}
\xyoption{arc}

\input xypic
\xyoption{all}
\xyoption{arc}

\usepackage{amsmath}
\usepackage{amssymb}

\usepackage{url}

\newtheorem{theo}{Theorem}[section]

\newtheorem{prop}[theo]{Proposition}

\newtheorem{lem}[theo]{Lemma}

\newtheorem{defi}[theo]{Definition}

\def\remark#1{{\refstepcounter{theo}\label{#1}\noindent\sc Remark  
\thesection.\arabic{theo} - }}

\def\equat{\refstepcounter{theo}$$~}
\def\endequat{\leqno{\boldsymbol{(\arabic{section}.\arabic{theo})}}~$$}

\newcounter{numero}[section]

\newcounter{fact}

\def\fact#1{\refstepcounter{fact}
\begin{quotation}
\noindent{\bf Fact \arabic{fact}.\label{#1}}} 
\def\endfact{\end{quotation}}

\newcounter{soussection}[section]

\newcounter{soussoussection}[soussection]

\def\CM{{\mathbb{C}}}

\def\FM{{\mathbb{F}}}

\def\NM{{\mathbb{N}}}

\def\QM{{\mathbb{Q}}}

\def\ZM{{\mathbb{Z}}}

\def\Arm{{\mathrm{A}}}
\def\Brm{{\mathrm{B}}}
\def\Crm{{\mathrm{C}}}
\def\Drm{{\mathrm{D}}}
\def\Erm{{\mathrm{E}}}

\def\a{\alpha}

\def\G{\Gamma}

\def\D{\Delta}
\def\e{\varepsilon}

\def\l{\lambda}

\def\m{\mu}

\def\s{\sigma}

\def\t{\tau}
\def\x{\xi}

\def\z{\zeta}

\def\EC{{\mathcal{E}}}

\def\MC{{\mathcal{M}}}

\def\SC{{\mathcal{S}}}

\def\ZC{{\mathcal{Z}}}

\def\Gb{{\mathbf G}}
\def\Hb{{\mathbf H}}

\def\Lb{{\mathbf L}}
\def\Mb{{\mathbf M}}

\def\Ob{{\mathbf O}}
\def\Pb{{\mathbf P}}
\def\Qb{{\mathbf Q}}

\def\Sb{{\mathbf S}}
\def\Tb{{\mathbf T}}
\def\Ub{{\mathbf U}}

\def\Zb{{\mathbf Z}}

\def\ib{{\mathbf i}}

\def\nb{{\mathbf n}}

\def\pb{{\mathbf p}}

\def\fba{{\bar{f}}}

\def\sba{{\bar{s}}}

\def\Gbt{{\tilde{\Gb}}}

\def\Mbt{{\tilde{\Mb}}}

\def\Sbt{{\tilde{\Sb}}}

\def\Gbh{{\hat{\Gb}}}

\def\Lbh{{\hat{\Lb}}}
\def\Mbh{{\hat{\Mb}}}

\def\Pbh{{\hat{\Pb}}}
\def\Qbh{{\hat{\Qb}}}

\def\Gbo{{\overline{\Gb}}}

\def\Lbo{{\overline{\Lb}}}
\def\Mbo{{\overline{\Mb}}}

\def\Pbo{{\overline{\Pb}}}
\def\Qbo{{\overline{\Qb}}}

\def\Sbo{{\overline{\Sb}}}

\def\ad{\mathop{\mathrm{ad}}\nolimits}

\def\Class{\mathop{\mathrm{Class}}\nolimits}

\def\Cus{{\mathrm{Cus}}}
\def\deg{\mathop{\mathrm{deg}}\nolimits}

\def\Id{\mathop{\mathrm{Id}}\nolimits}

\def\Im{\mathop{\mathrm{Im}}\nolimits}

\def\Irr{\mathop{\mathrm{Irr}}\nolimits}
\def\ker{\mathop{\mathrm{ker}}\nolimits}
\def\Ker{\mathop{\mathrm{Ker}}\nolimits}

\def\mod{\mathop{\mathrm{mod}}\nolimits}

\def\Res{\mathop{\mathrm{Res}}\nolimits}

\def\uni{{\mathrm{uni}}}

\def\to{\rightarrow}
\def\longto{\longrightarrow}

\def\fonction#1#2#3#4#5{\begin{array}{rccc}
{#1} : & {#2} & \longto & {#3} \\
& {#4} & \longmapsto & {#5} 
\end{array}}

\def\ci{\circ}

\def\incl{\hspace{0.05cm}{\subset}\hspace{0.05cm}}

\def\ql{{\QM_\el}}
\def\qlb{{\overline{\QM}_\el}}

\def\DS{\displaystyle}
\def\SS{\scriptstyle}
\def\SSS{\scriptscriptstyle}

\def\finl{~$\SS \square$}
\def\el{\ell}
\def\infspe{\hspace{0.1em}\mathop{\preccurlyeq}\nolimits\hspace{0.1em}}

\def\lexp#1#2{\kern\scriptspace\vphantom{#2}^{#1}\kern-\scriptspace#2}
\def\le{\hspace{0.1em}\mathop{\leqslant}\nolimits\hspace{0.1em}}
\def\ge{\hspace{0.1em}\mathop{\geqslant}\nolimits\hspace{0.1em}}

\mathchardef\lllllll="3278
\def\SEC{$\lllllll$}
\mathchardef\inferieur="321E
\mathchardef\superieur="321F

\def\eqna{\begin{eqnarray*}}
\def\endeqna{\end{eqnarray*}}

\def\para{parabolic subgroup }

\def\levi{Levi subgroup }

\def\itemth#1{\item[$({\mathrm{#1}})$]}

\def\gfp{{\FM_{\! p}}}
\def\gfq{{\FM_{\! q}}}

\newcommand\CHEVIE{{\tt CHEVIE}}
\newcommand\GAP{{\tt GAP}}
\newcommand\GAPtrois{{\tt GAP3}}

\begin{document}


\baselineskip=15pt

\title{Computational proof of the Mackey formula for $q > 2$}
\author{\sc C\'edric Bonnaf\'e \& Jean Michel}
\address{\noindent 
Laboratoire de Math\'ematiques de Besan\c{c}on (CNRS: UMR 6623), 
Universit\'e de Franche-Comt\'e, 16 Route de Gray, 25030 Besan\c{c}on
Cedex, FRANCE} 
\address{\noindent 
Institut de Math\'ematiques de Jussieu (CNRS: UMR 7586), 
Universit\'e Paris VII, 175 Rue du Chevaleret, 75013 Paris, FRANCE} 

\makeatletter
\email{cedric.bonnafe@univ-fcomte.fr}
\urladdr{www-math.univ-fcomte.fr/pp\_Annu/CBONNAFE/}

\makeatother

\email{jmichel@math.jussieu.fr}
\urladdr{www.math.jussieu.fr/~jmichel}


\date{\today}
\maketitle

\bigskip

\begin{abstract}
Let $\Gb$ be a connected reductive group defined over a finite field with $q$ elements. 
We prove that the Mackey formula for the Lusztig induction and restriction 
holds in $\Gb$ whenever $q>2$ or $\Gb$ does not have a component of type 
$\Erm$.
\end{abstract}

\bigskip

Let $\Gb$ be a connected reductive group defined over an algebraic closure 
$\FM$ of a finite field of characteristic $p > 0$ and let $F : \Gb \to \Gb$ 
be a Frobenius endomorphism endowing $\Gb$ with an $\gfq$-structure, 
where $q$ is a power of $p$ and $\gfq$ is the finite subfield of $\FM$ 
of cardinal $q$. By {\it the Mackey formula for Lusztig induction and restriction}, 
we mean the following formula
$$\lexp{*}{R}_{\Lb \incl \Pb}^\Gb \ci R_{\Mb \incl \Qb}^\Gb = 
\sum_{\SSS{g \in \Lb^F\backslash\SC_\Gb(\Lb,\Mb)^F/\Mb^F}} 
R_{\Lb \cap \lexp{g}{\Mb} \incl \Lb \cap \lexp{g}{\Qb}}^\Lb \ci 
\lexp{*}{R}_{\Lb \cap \lexp{g}{\Mb} \incl \Pb \cap \lexp{g}{\Mb}}^{\lexp{g}{\Mb}}
\ci (\ad g)_\Mb.\leqno{(\MC_{\Gb,\Lb,\Pb,\Mb,\Qb})}$$
Here, $\Pb$ and $\Qb$ are two parabolic subgroups of $\Gb$, $\Lb$ and $\Mb$ are $F$-stable 
Levi complements of $\Pb$ and $\Qb$ respectively, $R_{\Lb \incl \Pb}^\Gb$ and 
$\lexp{*}{R}_{\Lb \incl \Pb}^\Gb$ denote respectively the Lusztig induction and 
restriction maps, $(\ad g)_\Mb$ is the map between class functions on $\Mb^F$ and 
$\lexp{g}{\Mb^F}$ induced by conjugacy by $g$, and $\SC_\Gb(\Lb,\Mb)$ is the set of 
elements $g  \in \Gb$ such that $\Lb$ and $\lexp{g}{\Mb}$ have a common maximal torus.

It is conjectured that the Mackey formula always holds. This paper is a contribution 
towards a solution to this conjecture. Our aim is 
to prove the last two lines of the following theorem:

\bigskip

\noindent{\bf Theorem.} 
{\it Assume that one of the following holds:
\begin{itemize}
\itemth{1} $\Pb$ and $\Qb$ are $F$-stable (Deligne \cite[Theorem 2.5]{luspa}).

\itemth{2} $\Lb$ or $\Mb$ is a maximal torus of $\Gb$ (Deligne and 
Lusztig \cite[Theorem 7]{delu}).

\itemth{3} $q > 2$.

\itemth{4} $\Gb$ does not contain an $F$-stable quasi-simple component 
of type $\lexp{2}{\Erm_6}$, $\Erm_7$ or $\Erm_8$.
\end{itemize}
Then the Mackey formula $(\MC_{\Gb,\Lb,\Pb,\Mb,\Qb})$ holds.}

\bigskip

While the proofs of (1) and (2) work in full generality and are pretty elegant, 
our proof of (3) and (4) is as ugly as possible. It follows an induction argument suggested by 
Deligne and Lusztig \cite[Proof of Theorems 6.8 and 6.9]{DL} (and improved in 
\cite{bonnafe q} and \cite{bonnafe a}) that shows that if the semisimple elements 
of $\Gb^F$ satisfy some strange properties (see Proposition \ref{strange} for 
the list of properties) then the Mackey formula holds: then, checking 
Proposition \ref{strange} in cases (3) and (4) of the above theorem is done 
by a case-by-case analysis together with computer calculations using 
the \CHEVIE~ package (in \GAPtrois).

Even when a proof of some important result requires a case-by-case analysis, 
one might expect to get some interesting intermediate mathematical results: 
this is not even the case in this paper. The interest of this paper is of two 
kinds: the result (not its proof) and the development of the \CHEVIE~ package 
for computing with (semisimple elements of) algebraic groups. This extension 
of the \CHEVIE~package, 
together with some application to our problems, is presented in an appendix 
at the end of this paper.

\bigskip

\noindent{\sc Remark - } In fact, our proof shows that, if the Mackey formula 
$(\MC_{\Gb,\Lb,\Pb,\Mb,\Qb})$ holds whenever $(\Gb,F)$ is the semisimple and simply-connected 
group of type $\lexp{2}{\Erm}_6(2)$ and $\Mb$ is of type $\Arm_2 \times \Arm_2$, 
then the Mackey formula holds in general (see Remark \ref{2E6}).\finl

\bigskip

\section{Notation, recollection}

\medskip

\subsection*{Algebraic groups}
We fix a prime number $p$, an algebraic closure $\FM$ of the finite field with 
$p$ elements $\gfp$, a power $q$ of $p$ and a connected reductive group $\Gb$ (over $\FM$) 
endowed with an $\gfq$-structure determined by a Frobenius endomorphism $F : \Gb \to \Gb$ 
(here, $\gfq$ denotes the subfield of $\FM$ with $q$ elements). We also fix a 
pair $(\Gb^*,F^*)$ dual to $(\Gb,F)$ and we denote by $\pi : \Gbt^* \to \Gb^*$ 
the simply-connected covering of the derived subgroup of $\Gb^*$. Then there exists 
a unique $\gfq$-Frobenius endomorphism $F^* : \Gbt^* \to \Gbt^*$ such that 
$\pi$ is defined over $\gfq$.

In this paper, if $\Hb$ is an algebraic group, we denote by $\Hb^\circ$ its 
neutral component. If $\Ub$ denotes the unipotent radical of $\Hb$, a {\it Levi complement} 
of $\Hb$ is a subgroup $\Lb$ of $\Hb$ such that $\Hb = \Lb \ltimes \Ub$. 
We shall define a {\it Levi subgroup} of $\Gb$ to be a Levi complement of 
some parabolic subgroup of $\Gb$. The centre of $\Hb$ will be denoted by $\Zb(\Hb)$ 
and we set $\ZC(\Hb)=\Zb(\Hb)/\Zb(\Hb)^\circ$. If $g \in \Hb$, the order of $g$ will be denoted by $o(g)$. 

If $\Lb$ is a Levi subgroup of 
$\Gb$, then the morphism $h_\Lb^\Gb : \ZC(\Gb) \to \ZC(\Lb)$ is surjective 
(see \cite[Lemma 1.4]{dlm}) and its kernel has been completely computed in 
\cite[Proposition 2.8 and Table 2.17]{bonnafe reg}). If $\Mb$ is another Levi subgroup 
of $\Gb$, we denote by $\SC_\Gb(\Lb,\Mb)$ the set of elements $g \in \Gb$ such that 
$\Lb$ and $\lexp{g}{\Mb}$ have a common maximal torus. Recall that this implies 
that $\Lb \cap \lexp{g}{\Mb}$ is a Levi complement of $\Lb \cap \lexp{g}{\Qb}$, 
as well as a Levi complement of $\Pb \cap \lexp{g}{\Mb}$. 


If $\Zb$ is an $F$-stable subgroup of the centre $\Zb(\Gb)$ of $\Gb$, 
we also fix a pair $((\Gb/\Zb)^*,F^*)$ dual to $(\Gb/\Zb,F)$ 
and we denote by $\pi_\Zb : \Gbt^* \to (\Gb/\Zb)^*$ the induced morphism: note that 
it is defined over $\gfq$. There exists a unique morphism $\t_\Zb : (\Gb/\Zb)^* \to \Gb^*$ 
such that $\pi=\t_\Zb \circ \pi_\Zb$: it is also defined over $\gfq$. Finally, we set 
$\Zb^* = \Ker \pi_\Zb$: it is an $F^*$-stable subgroup of $\Zb(\Gbt^*)$, which 
should not be confused with a dual (in any sense) of $\Zb$.

We also recall the following definition:

\bigskip

\begin{defi}
A semisimple element $s \in \Gb^*$ is said to be {\bfit isolated} 
(respectively {\bfit quasi-isolated}) if its connected centralizer $C_{\Gb^*}^\circ(s)$ 
(respectively its centralizer $C_{\Gb^*}(s)$) is not contained in a proper 
Levi subgroup of $\Gb$. 
\end{defi}

\bigskip

\subsection*{Class functions}
We fix a prime number $\ell \neq p$ and we denote by $\qlb$ an algebraic closure 
of the $\ell$-adic field $\ql$. If $\G$ is a finite group, the $\qlb$-vector space of 
class functions $\G \to \qlb$ is denoted by $\Class(\G)$. This vector space 
is endowed with the canonical scalar product $\langle , \rangle_\G$, for which 
the set of irreducible characters $\Irr \G$ of $\G$ is an orthonormal basis.

If $\Lb$ is an $F$-stable Levi complement of a parabolic subgroup $\Pb$ of $\Gb$, 
let $R_{\Lb \incl \Pb}^\Gb : \Class(\Lb^F) \to \Class(\Gb^F)$ and 
$\lexp{*}{R}_{\Lb \incl \Pb}^\Gb : \Class(\Gb^F) \to \Class(\Lb^F)$ 
denote respectively the Lusztig induction and restriction maps. 
They are adjoint with respect to the scalar products $\langle,\rangle_{\Lb^F}$ and 
$\langle , \rangle_{\Gb^F}$. If $g \in \Gb^F$, we denote by 
$(\ad g)_\Lb : \Class(\Lb^F) \to \Class(\lexp{g}{\Lb^F})$, 
$\l \mapsto (\lexp{g}{\l} : l \mapsto \l(g^{-1}lg))$.

If $s \in \Gb^F$ is a semisimple element and if $f \in \Class(\Gb^F)$, we define 
$$\fonction{d_s^\Gb f}{C_\Gb^\circ(s)^F}{\qlb}{u}{\begin{cases} 
f(su) & \text{if $u$ is unipotent,}\\
0 & \text{otherwise.}
\end{cases}}$$
Note that $d_s^\Gb f \in \Class(C_\Gb^\circ(s)^F)$, so that we have defined a $\qlb$-linear map
$$d_s^\Gb : \Class(\Gb^F) \longto \Class(C_\Gb^\circ(s)^F).$$

\bigskip

\subsection*{Tori over finite fields}
If $\Sb$ is a torus defined over $\gfq$, we denote by $X(\Sb)$ 
(respectively $Y(\Sb)$) the lattice of rational characters $\Sb \to \FM^\times$ (respectively 
of one-parameter subgroups $\FM^\times \to \Sb$). 
Let $\langle,\rangle_\Sb : X(\Sb) \times Y(\Sb) \to \ZM$ 
denote the canonical perfect pairing. 

If moreover $\Sb$ is defined over $\gfq$, 
with corresponding Frobenius endomorphism $F : \Sb \to \Sb$, then there exists a unique 
automorphism $\phi : Y(\Sb) \to Y(\Sb)$ of {\it finite order} such that 
$F(\l)= q \phi(\l)$ for all $\l \in Y(\Sb)$. The characteristic polynomial of 
$\phi$ will be denoted by $\chi_{\Sb,F}$: since $\phi$ has finite order, 
$\chi_{\Sb,F}$ is a product of cyclotomic polynomials. Note that
\equat\label{deg dim}
\deg \chi_{\Sb,F} = \dim \Sb
\endequat
and that \cite[Proposition 13.7 (ii)]{dmbook}
\equat\label{structure tore}
\Sb^F \simeq Y(\Sb)/(F-\Id_{Y(\Sb)})(Y(\Sb))\qquad\text{and}\qquad |\Sb^F|=\chi_{\Sb,F}(q).
\endequat
If $m$ is a non-zero natural number, we denote by $\Phi_m$ the $m$-th cyclotomic polynomial. 
We shall recall here the notions of $\Phi_m$-torus, as defined in \cite[Definition 3.2]{BM}: 

\bigskip

\begin{defi}[Brou\'e-Malle-Michel]
We say that $(\Sb,F)$ is a {\bfit $\Phi_m$-torus} if the characteristic polynomial of $\phi$ 
is a power of $\Phi_m$.

If $\Sb'$ is an $F$-stable subtorus of $\Sb$, we say that 
$(\Sb',F)$ is a {\bfit Sylow $\Phi_m$-subtorus} of $(\Sb,F)$ if $\chi_{\Sb',F}$ is exactly 
the highest power of $\Phi_m$ dividing $\chi_{\Sb,F}$.
\end{defi}

\bigskip

Recall \cite[Theorem 3.4 (1)]{BM} 
that there always exists a Sylow $\Phi_m$-subtorus, that it is unique and that, 
if $(\Sb,F)$ is itself a $\Phi_m$-torus (with $\chi_{\Sb,F}=\Phi_m^r$), then 
\cite[Proposition 3.3 (3)]{BM}
\equat\label{ordre tore}
\Sb^F \simeq \bigl(\ZM/\Phi_m(q)\ZM\bigr)^r.
\endequat
It follows from the definition that $(\Sb,F)$ is a $\Phi_1$-torus if and only if 
$(\Sb,F)$ is a split torus.

\bigskip

\begin{lem}\label{dim 2}
If $\dim \Sb = 2$, then $\Sb^F$ is isomorphic to one of the following 
groups
$$\bigl(\ZM/(q-1)\ZM\bigr)^2,\quad \ZM/(q-1)\ZM \times \ZM/(q+1)\ZM,
\quad \bigl(\ZM/(q+1)\ZM\bigr)^2,$$
$$\ZM/(q^2-1)\ZM,\quad \ZM/(q^2+q+1)\ZM \quad \text{or}\quad
\ZM/(q^2-q+1)\ZM.$$
\end{lem}

\bigskip

\begin{proof}
Since the only cyclotomic polynomials of degree $\le 2$ are $\Phi_1$, $\Phi_2$, 
$\Phi_3$ and $\Phi_6$, it follows from \ref{deg dim} that $(\Sb,F)$ is a $\Phi_m$-torus 
for $m \in \{1,2,3,6\}$ or that $\chi_{\Sb,F}=\Phi_1 \Phi_2$. In the first case, 
the result follows from \ref{ordre tore}. 

So let us examine now the case where $\chi_{\Sb,F}=\Phi_1\Phi_2$. So $\phi^2=\Id_{Y(\Sb)}$. 
For $m \in \{1,2\}$, 
let $\Sb_m$ denote the Sylow $\Phi_m$-subtorus of $\Sb$ and let $\l_m$ be a generator 
of $Y(\Sb_m)=\{\l \in Y(\Sb)~|~\Phi_m(\phi)(\l)=0\}$. 
So $(\l_1,\l_2)$ is a basis of $\QM \otimes_\ZM Y(\Sb)$. Note that $Y(\Sb)/Y(\Sb_m)$ 
is torsion-free. Two cases may occur:

\medskip

$\bullet$ If $(\l_1,\l_2)$ is a $\ZM$-basis of $Y(\Sb)$. Then $\Sb \simeq \Sb_1 \times \Sb_2$ 
and $\Sb^F \simeq \Sb_1^F \times \Sb_2^F \simeq \ZM/(q-1)\ZM \times \ZM/(q+1)\ZM$. 

\medskip

$\bullet$ If $(\l_1,\l_2)$ is not a basis of $Y(\Sb)$, let $\l=(\l_1+\l_2)/2 \in \QM \otimes_\ZM Y(\Sb)$ 
and let $\m=\phi(\l)=(\l_1-\l_2)/2 \in \QM \otimes_\ZM Y(\Sb)$. First of all, let us show that 
$\l \in Y(\Sb)$. Indeed, there exists $a$ and $b$ in $\QM$ such that 
$\eta_0=a_0\l_1 + b_0 \l_2 \in Y(\Sb)$ 
and $(a_0,b_0) \not\in \ZM \times \ZM$. Since $\eta_0 +\phi(\eta_0) \in Y(\Sb_1)$ and 
$\eta_0-\phi(\eta_0) \in Y(\Sb)$, we get that $2a_0 \in \ZM$ and $2b_0 \in \ZM$. 
By replacing $a_0$ and $b_0$ by $a_0-a_0'$ and $b_0-b_0'$ with $b_0$, $b_0' \in \ZM$, 
we may (and we will) assume that $a_0$, $b_0 \in \{0,1/2\}$. 
But $\l_1/2\not\in Y(\Sb)$ and $\l_2/2\not\in Y(\Sb)$. 
So $(a_0,b_0)=(1/2,1/2)$. So $\l \in Y(\Sb)$. 

Similarly, $\m=\phi(\l) \in Y(\Sb)$. Now, if $\eta \in Y(\Sb)$, then there exists 
$a$, $b \in \ZM$ such that $\eta + \phi(\eta)=a \l_1$ and $\eta-\phi(\eta)=b\l_2$. 
Therefore, $\eta = (a \l_1 + b \l_2)/2 = (a+b)\l + (a-b) \m$. So $(\l,\m)$ is 
a $\ZM$-basis of $Y(\Sb)$. In this basis, the matrix representing $F$ is
$$\begin{pmatrix} 0 & q \\ q & 0 \end{pmatrix}.$$
It then follows from \ref{structure tore} that $\Sb^F$ is cyclic of order 
$q^2-1$. 
\end{proof}

\bigskip

\section{A property of quasi-isolated semisimple elements\label{sec:semisimple}}

\medskip

The aim of this section is to prove the following proposition, from 
which the Mackey formula in the cases (3) and (4) of the Theorem will be deduced.

\bigskip

\begin{prop}\label{strange}
Let $\Mb$ be an $F$-stable Levi complement 
of a parabolic subgroup $\Qb$ of $\Gb$. Assume that the quadruple $(\Gb,F,\Mb,\Qb)$ 
satisfies all of the following conditions:
\begin{itemize}
\itemth{P1} $\Gb$ is semisimple and simply-connected and $F$ permutes transitively 
the quasi-simple components of $\Gb$;

\itemth{P2} The quasi-simple components of $\Gb$ are not of type $\Arm$;

\itemth{P3} $\Mb$ is not a maximal torus of $\Gb$ and $\Mb\not= \Gb$;

\itemth{P4} There exists an $F$-stable unipotent class of $\Mb$ which supports an $F$-stable 
cuspidal local system;

\itemth{P5} $\Qb$ is not contained in an $F$-stable proper \para of $\Gb$;

\itemth{P6} For every $F$-stable subgroup $\Zb$ of $\Zb(\Gb) \cap \Zb(\Mb)^\circ$, 
there exists a semisimple element $s \in (\Mb/\Zb)^{*F^*}$ which 
is quasi-isolated in $(\Mb/\Zb)^*$ and $(\Gb/\Zb)^*$ and 
such that, for every $z \in \pi_\Zb(\Gbt^{*F^*}) \cap \Zb((\Mb/\Zb)^*)^{F^*}$, 
$s$ and $sz$ are conjugate in $(\Gb/\Zb)^{*F^*}$.
\end{itemize}
Then $\Gb$ is quasi-simple, the pair $(\Gb,F)$ is of type $\lexp{2}{\Erm_6}$, $q=2$ and $\Mb$ is of type 
$\Arm_2 \times \Arm_2$. 
\end{prop}

\bigskip

The rest of this section is devoted to the proof of this proposition. 
This will be done through a case-by-base analysis, relying on some 
computer calculation using the \CHEVIE~ package (in \GAPtrois). Before starting 
the case-by-case analysis, we gather some consequences of properties 
(Pk), $1 \le k \le 6$, that hold in all groups.

So, from now on, and until the end of this section \ref{sec:semisimple}, we fix an $F$-stable 
Levi complement $\Mb$ of a parabolic subgroup $\Qb$ of $\Gb$ such that the 
quadruple $(\Gb,F,\Mb,\Qb)$ satisfies the statements 
(P1), (P2), (P3), (P4), (P5) and (P6) of the Proposition \ref{strange}. Using (P1), 
let us write 
$$\Gb=\underbrace{\Gb_0 \times \dots \times \Gb_0}_{d\text{ times}},$$
where $\Gb_0$ is a semisimple, simply-connected and quasi-simple 
group defined over $\FM_{\! q^d}$ and $F$ permutes transitively the 
quasi-simple components of $\Gb$. In particular, $F^d$ stabilizes $\Gb_0$ and 
$\Gb^F \simeq \Gb_0^{F^d}$. 
Let $\Mb_0$ denote the $F^d$-stable \levi of $\Gb_0$ such that 
$\Mb=\Mb_0 \times \lexp{F}{\Mb_0} \times \dots \times \lexp{F^{d-1}}{\Mb_0}$. 
We denote by $\pi_0 : \Gbt_0^* \to \Gb_0^*$ the restriction of $\pi$ 
to the first component. 

We write 
$$\Gb^*=\underbrace{\Gb_0^* \times \dots \times \Gb_0^*}_{d\text{ times}}, \quad
\Gbt^*=\underbrace{\Gbt_0^* \times \dots \times \Gbt_0^*}_{d\text{ times}}$$
$$\Mb^*=\Mb_0^* \times \lexp{F^*}{\Mb_0^*} \times \dots \times 
\lexp{F^{*d-1}}{\Mb_0^*},$$
$$\Mbt^*=\pi^{-1}(\Mb^*)\qquad\text{and}\qquad \Mbt_0^*=\pi_0^{-1}(\Mb_0^*).$$
If $Z \incl \Zb(\Gb_0) \cap \Zb(\Mb_0)^\circ$ is $F^d$-stable, we set 
$\Zb=Z \times \lexp{F}{Z} \times \dots \times \lexp{F^{d-1}}{Z}$, 
$S_Z=\Zb((\Mb/\Zb)^*)^{\circ F^*}$, $S_Z'=S_Z \cap \pi_\Zb(\Zb(\Mbt^*)^{\circ F^*})$, 
$e_Z=|S_Z|/|S_Z'|$ and we denote by $s_Z$ a semisimple element of 
$(\Mb/\Zb)^{*F^*}$ which is quasi-isolated in $(\Mb/\Zb)^*$ and 
$(\Gb/\Zb)^*$ and such that, for every $z \in S_Z'$, $s_Z z$ and $s_Z$ 
are conjugate in $(\Gb/\Zb)^{*F^*}$. Note that such an element exists 
by (P6). If $Z=1$, we set $s_Z=s$, $S_Z=S$, $S_Z'=S'$ and $e_Z=e$ for simplification. 
We denote by $s_0$ the projection of $s$ on the first component $\Gb_0$ of $\Gb$. 
Recall that $\Zb^*$ is the kernel 
of $\pi_\Zb$: we denote by $\Zb_0^*$ the projection of $\Zb^*$ on the first component $\Gb_0^*$. 

Note that, if $Z=\Zb(\Gb_0)$ (which might happen only if $\Zb(\Gb_0) \subseteq \Zb(\Mb_0)^\circ$), 
then $\Zb=\Zb(\Gb)$, $(\Gb/\Zb)^*=\Gbt^*$, $\Zb^*=1$ and $\pi_{\Zb(\Gb_0)}$ is the identity. 
Then:

\bigskip

\begin{lem}\label{fourre-tout}
The properties $(${\rm Pk}$)$, $1 \le k \le 6$, have the following consequences:
\begin{itemize}
\itemth{a} There exists an $F^d$-stable unipotent class of $\Mb_0$ which supports an $F^d$-stable 
cuspidal local system.

\itemth{b} $\Zb(\Mb)^\circ$ is not an $F$-split torus (for the action of $F$). 

\itemth{c} $e_Z=|H^1(F^*,\Zb^* \cap \Ker h_{\Mbt^*}^{\Gbt^*})|=
|H^1(F^{*d},\Zb_0^* \cap \Ker h_{\Mbt_0^*}^{\Gbt_0^*})|$. 
In particular, $e_{\Zb(\Gb_0)}=1$ and 
$e_1=|H^1(F^{*d}, \Ker h_{\Mbt_0^*}^{\Gbt_0^*})|$.

\itemth{d} $S_Z$ contains an element of order $\ge \max(q^d-1,q+1)$.

\itemth{e} All the elements of $S_Z'$ have order dividing the order of $s_Z$.

\itemth{f} If $\Mb_0$ is of type $\Brm$, $\Crm$ or $\Drm$, then $p \neq 2$.

\itemth{g} If $\dim \Zb(\Mb_0^*)=1$, then $S_Z$ is isomorphic to $\ZM/(q^d-1)\ZM$ or $\ZM/(q^d+1)\ZM$.

\itemth{h} If $\dim \Zb(\Mb_0^*)=2$, then $S_Z$ is isomorphic to one of the following 
groups
$$\bigl(\ZM/(q^d-1)\ZM\bigr)^2,\quad \ZM/(q^d-1)\ZM \times \ZM/(q^d+1)\ZM,
\quad \bigl(\ZM/(q^d+1)\ZM\bigr)^2,$$
$$\ZM/(q^{2d}-1)\ZM,\quad \ZM/(q^{2d}+q^d+1)\ZM \quad \text{or}\quad
\ZM/(q^{2d}-q^d+1)\ZM.$$
\end{itemize}
\end{lem}

\begin{proof}[Proof of lemma \ref{fourre-tout}]
(a) follows immediately from (P4). (b) follows from (P5) and the following well-known result:
\begin{quotation}
{\small 
\begin{lem}\label{deploiement}
Let $\Pb$ be a parabolic subgroup of $\Gb$ and let $\Lb$ be a Levi complement 
of $\Pb$. Assume that $\Lb$ is $F$-stable and that $\Zb(\Lb)^\circ$ is $F$-split. 
Then $\Pb$ is $F$-stable.
\end{lem}

\medskip

\begin{proof}[Proof of Lemma \ref{deploiement}]
Let $\Tb$ be an $F$-stable maximal torus of $\Lb$. Let $\Phi \subset X(\Tb)$ denote 
the root system of $\Gb$ with respect to $\Tb$. If $\a \in \Phi$, we denote by $\Ub_\a$ 
the associated one-parameter unipotent subgroup. If $\l \in Y(\Tb)$, we denote by $\Pb(\l)$ 
the subgroup of $\Gb$ generated by $\Tb$ and the $\Ub_\a$'s such that $\langle \a,\l \rangle \ge 0$. 
Then $\Pb(\l)$ is a parabolic subgroup and $F(\Pb(\l))=\Pb(F(\l))$.

Since $\Pb$ is a parabolic subgroup of $\Gb$ admitting $\Lb$ as a Levi complement, 
there exists $\l \in Y(\Zb(\Lb)^\circ) \subset Y(\Tb)$ such that $\Pb=\Pb(\l)$. 
Now, $F(\Pb)=\Pb(F(\l))$ and $F(\l)=q \l$ because $F$ is split on $\Zb(\Lb)^\circ$. 
So $F(\Pb)=\Pb$, as expected.
\end{proof}}
\end{quotation}

\medskip


(c) Let $K=\Ker \pi_\Zb \cap \Zb(\Mbt^*)^\circ = \Zb^* \cap \Ker h_{\Mbt^*}^{\Gbt^*}$. 
From the natural exact sequence 
$1 \to K \to \Zb(\Mbt^*)^\circ \to \Zb((\Mb/\Zb)^*)^\circ \to 1$, we deduce an exact sequence 
of cohomology groups 
$$ 1 \longto K^{F^*} \longto \Zb(\Mbt^*)^{\circ F^*} \to \Zb((\Mb/\Zb)^*)^{\circ F^*} \longto 
H^1(F^*,K) \longto H^1(F^*,\Zb(\Mbt^*)^\circ)=1.$$
The first equality in (c) then follows immediately. The second is straightforward.

\medskip

(d) Note that $\Phi_m(q) \ge q-1\ge 1$ if $m \ge 1$ and 
$\Phi_m(q) \ge q+1 \ge 3$ if $m \ge 2$ (note, however, that $\Phi_6(2)=\Phi_2(2)=3$). 
Let $\Sb_Z=\Zb((\Mb_0/Z)^*)^\circ$. We have $S_Z \simeq \Sb_Z^{F^{*d}}$. 
Since $\Mb \neq \Gb$, we have that $\dim \Sb_Z \ge 1$. So there exists 
$m \ge 1$ such that $\Phi_m$ divides $\chi_{\Sb_Z,F^{*d}}$. 
By \ref{ordre tore}, there exists an element of $S_Z$ of order $\Phi_m(q^d) \ge q^d-1$. 

If $d \ge 2$, then $q^d - 1 \ge q+1$. If $d=1$, then $(\Sb_Z,F^*)$ is not split 
by (b), so there exists $m \ge 2$ such that $\Phi_m$ divides $\chi_{\Sb_Z,F^*}$. 
So, again by \ref{ordre tore}, there exists an element of $S_Z$ of order $\Phi_m(q) \ge q+1$. 

\medskip

(e) Let $z \in S_Z'$. Let $d$ denote the order of $s_Z$. 
Since $s_Zz$ and $s_Z$ are conjugate in $(\Gb/Z)^*$, we have $(s_Zz)^d=1$. 
But $s_Zz=zs_Z$ and $s_Z^d=1$, so $z^d=1$.

\medskip

(f) Assume that $\Mb_0$ is of type $\Brm$, $\Crm$ or $\Drm$ and that $p=2$. 
Then $\Zb(\Mb)=\Zb(\Mb)^\circ$ and $\Zb(\Mb^*)=\Zb(\Mb^*)^\circ$. 
So $\Zb(\Gb) \incl \Zb(\Mb)^\circ$. 
Therefore, by \cite[Example 4.8]{bonnafe quasi}, $s_{\Zb(\Gb_0)}$ is central in $\Mbt^*$. 
Moreover, $S_{\Zb(\Gb_0)}=S_{\Zb(\Gb_0)}'$ by (c), and $s_{\Zb(\Gb_0)} \in S_{\Zb(\Gb_0)}$. 
So, by Lemma \ref{fourre-tout} (d), there exists an element $z$ of $S_{\Zb(\Gb_0)}'$ 
different from $1$. 
If $s_{\Zb(\Gb_0)}=1$, then $s_{\Zb(\Gb_0)}$ and $s_{\Zb(\Gb_0)} z$ 
are not conjugate in $\Gbt^{*F^*}$, contradicting (P6). 
If $s_{\Zb(\Gb_0)}\not=1$, then $s_{\Zb(\Gb_0)}$ and 
$s_{\Zb(\Gb_0)} s_{\Zb(\Gb_0)}^{-1}=1$ are not conjugate in $\Gbt^{*F^*}$, 
contradicting again (P6). 

\medskip

(g) follows from \ref{ordre tore} and (h) follows from Lemma \ref{dim 2}. 
\end{proof}

\bigskip

We can now start our case-by-case analysis, that will be done as a long sequence 
of lemmas.

\bigskip

\noindent{\bf Lemma FG.} 
{\it The group $\Gb_0$ is not of type $F_4$ or $G_2$.}

\bigskip

\begin{proof}[Proof of Lemma FG]
Assume first that $\Gb_0$ is of type $G_2$. 
Then by Lemma \ref{fourre-tout} (a) and \cite[\SEC 15.5]{luicc}, 
we get that $\Mb_0 = \Gb_0$ or that $\Mb_0$ is a torus. 
This contradicts (P3).

\medskip

Assume now that $\Gb_0$ is of type $F_4$. Then by Lemma \ref{fourre-tout} (a) and \cite[\SEC 15.4]{luicc}, 
we get that $\Mb_0 = \Gb_0$ or that $\Mb_0$ is a torus or that $p=2$ and $\Mb_0$ is of type $\Brm_2$. 
This contradicts (P3) and Lemma \ref{fourre-tout} (f).
\end{proof}

\bigskip

\noindent{\bf Lemma BCD2.} 
{\it If $\Gb_0$ is of type $\Brm$, $\Crm$ or $\Drm$, then $p > 2$.}

\bigskip

\begin{proof}[Proof of Lemma BCD2]
Assume that $\Gb$ is of type $\Brm$, $\Crm$ or $\Drm$ and that $p=2$. Then $s=1$ (see 
\cite[Example 4.8]{bonnafe quasi}), $e=1$ (since $\Zb(\Gbt^*)=1$), 
and, if $z$ denotes an element of $S=S'$ different from $1$ 
(such an element exists by Lemma \ref{fourre-tout} (d)), 
then $s$ and $sz$ are not conjugate in $\Gb^{*F^*}$.\finl
\end{proof}

\bigskip

\noindent{\bf Lemma C.} 
{\it The group $\Gb_0$ is not of type $\Crm$.}

\bigskip

\begin{proof}[Proof of Lemma C]
Assume that $\Gb_0$ is of type $\Crm_n$ with $n \ge 2$. 
Note that $p\not=2$ by Lemma BCD2 (so $q \ge 3$). 
Since $\Gb_0^*$ is a special orthogonal group, we have $o(s)=1$ or $2$ (see 
\cite[Proposition 4.11]{bonnafe quasi}). 
Moreover, since $|\Zb(\Gbt^*)|=2$, we have $e \le 2$ by Lemma \ref{fourre-tout} (c). 
Therefore, if $S$ contains a non-trivial element $z$ of odd order, 
then $z \in S'$ and $s$ and $sz$ 
are not conjugate in $\Gb^*$. This shows that every element of $S$ is a $2$-element. 
 
On the other hand, if $S$ contains a non-trivial element 
$z$ of order greater than or equal to $8$, then 
$z^2 \in S'$ and $o(z^2) \ge 4 > 2$. So 
$s$ and $sz^2$ are not conjugate in $\Gb^*$, contradicting (P6). 
So, every element of $S$ has order $1$, $2$ or $4$.
In particular, by Lemma \ref{fourre-tout} (d), 
we get that $d=1$ and $q=3$, so $\Gb=\Sb\pb_{2n}(\FM)$. 
In particular, $\Gb$ is split. 

By (P4) and by \cite[\SEC 10.4]{luicc}, we have 
$\Gb^* \simeq \Sb\Ob_{2n+1}(\FM)$ and $\Mb^* \simeq \Sb\Ob_{2m+1}(\FM) 
\times (\FM^\times)^r$ with $n=m+r$ and $r \ge 1$. Note that $\Zb(\Mb^*)=(\FM^\times)^r$. 
Since every element of $S$ has order $1$, $2$ or $4$, this means 
that $F^*$ acts on $\Mb^*$ (through the previous isomorphism) via the following formula~:
$$F^*(\s,t_1,\dots,t_r)=(F^*(\s),t_1^{3\e_1},\dots,t_r^{3\e_r})$$
for every $\s \in \Sb\Ob_{2m+1}(\FM)$ and $t_j \in \FM^\times$. Here, 
$\e_j \in \{1,-1\}$. Moreover, by Lemma \ref{fourre-tout} (d), $S$ contains at least 
one element of order $4$, so there exists $j$ such 
that $\e_j=-1$. Let $i$ denote a fourth root of unity in $\FM^\times$. 
Then 
$$z=(\Id,1,\dots,1,\underbrace{i}_{j-\text{th} \atop \text{position}},1,\dots,1) \in 
\Sb\Ob_{2m+1}(\FM) \times (\FM^\times)^r \simeq \Mb^*$$
is an element of $S$. In particular, $z^2 \in S'$ (since $e \le 2$). 
Let us write $s=(s',\x_1,\dots,\x_r) \in \Mb^{*F^*}$ with $s' \in \Sb\Ob_{2m+1}(\FM)$ 
and $\x_i \in \FM^\times$. Since $s^2=1$, we have 
$s^{\prime 2}=1$ and $\x_i^2=1$. It is now easy to check that $s$ and 
$sz^2$ are not conjugate in $\Gb^*$.
\end{proof}

\bigskip

\noindent{\bf Lemma SO.} 
{\it There does not exist a subgroup $Z$ of $\Zb(\Gb_0) \cap \Zb(\Mb_0)^\circ$ 
such that $\Gb_0/Z \simeq \Sb\Ob_N(\FM)$ for some $N \ge 7$.}

\bigskip

\begin{proof}[Proof of Lemma SO]
Assume that $\Gb_0\simeq \Sb\pb\ib\nb_N(\FM)$, with $N \ge 7$, and that 
there exists a subgroup $Z$ of $\Zb(\Gb_0) \cap \Zb(\Mb_0)^\circ$ 
such that $\Gb_0/Z \simeq \Sb\Ob_N(\FM)$. Note that $p \ge 3$ 
by Lemma BCD2. Then, by (P3), by Lemma \ref{fourre-tout} (a) 
and by \cite[\SEC\SEC 10.6 and 14]{luicc}, we have 
$Z=\Zb(\Mb_0)^\circ \cap \Zb(\Gb_0)$. In particular, $Z$ is $F^d$-stable. 
Then $(\Gb_0/Z)^*$ is a special orthogonal or a symplectic group. 
Thus, $s_Z^2 =1$ by \cite[Example 4.10 and Proposition 4.11]{bonnafe quasi}. Moreover, by 
\cite[\SEC 10.6]{luicc} and \cite[Table 2.17]{bonnafe reg}, $e_Z=1$. 
Therefore, $S_Z'$ contains an element $z$ 
of order greater than or equal to $q+1 \ge 4$. So $s_Z$ and $s_Zz$ are not conjugate 
in $(\Gb/Z)^*$. This contradicts (P6).
\end{proof}

\bigskip

\noindent{\bf Corollary D4.} 
{\it The group $\Gb_0$ is not of type $\Drm_4$.}

\bigskip

\begin{proof}[Proof of Corollary D4] 
Assume that $\Gb_0$ is of type $\Drm_4$. Then, by (P3), by Lemma \ref{fourre-tout} (a) 
and by \cite[\SEC\SEC 10.6 and 14]{luicc}, we have that $\Mb_0$ is of type 
$\Arm_1 \times \Arm_1$. Now, let $Z=\Zb(\Gb_0) \cap \Zb(\Mb_0)^\circ$. Then $|Z|=2$ 
and $(\Gb_0/Z)^*$ is a special orthogonal group. This is impossible by 
Lemma SO.
\end{proof}

\bigskip

\noindent{\bf Lemma BD.} 
{\it The group $\Gb_0$ is not of type $B$ or $D$.}

\bigskip

\begin{proof}[Proof of Lemma BD]
Assume that $\Gb_0 \simeq \Sb\pb\ib\nb_N(\FM)$. 
By Lemma BCD2, we have $p \ge 3$. Moreover, by (P2) and 
by Lemma C (and since $\Sb\pb\ib\nb_N(\FM)\simeq\Sb\pb_4(\FM)$, 
we have that $N \ge 7$. Let $n$ denote the rank of $\Gb_0$. 
We have $n=[N/2]$. Then $\Gb_0$ is of type $\diamondsuit_n$, with 
$\diamondsuit \in \{B,D\}$. By Lemma \ref{fourre-tout} (a), by Lemma SO and 
by \cite[\SEC\SEC 10.6 and 14]{luicc}, $\Mb_0$ is of type 
$\diamondsuit_m \times (\Arm_1)^r$ with $m \ge 0$ and $n=m+2r$. 

\medskip

Note that $N \not= 8$ by Lemma D4, that $e \le 2$, and that $s^4=1$ 
(see \cite[Table 2]{bonnafe quasi}). So, by (P6), every element of $S$ 
has order dividing $8$. By Lemma \ref{fourre-tout} (d), this implies that 
$d\in \{1,2\}$ and $q\in \{3,7\}$. Assume first that $d=2$. 
Then necessarily $q=3$ (by Lemma \ref{fourre-tout} (d)) and $S \simeq (\ZM/8\ZM)^r$. 
If $r \ge 2$, then $S'$ contains an element $z$ of order $8$ (because $S'$ has index at most 
$2$ in $S$) and $s$ and $sz$ are not conjugate in $\Gb^*$: this contradicts (P6). 
So $r=1$, but then $e=1$ so $S=S'$ contains an element of order $8$: 
this contradicts again (P6). Therefore, $d=1$, $\Gb=\Gb_0$ and $\Mb=\Mb_0$ and $q \in \{3,7\}$. 

Let $Z$ be the subgroup of $\Zb(\Gb)$ of order $2$ such that 
$\Gb/Z \simeq \Sb\Ob_N(\FM)$. Then $Z$ is $F$-stable. Write $\Gbo=\Gb/Z$ 
and $\Mbo=\Mb/Z$ and let $\sba$ be an element of $\Gbo^*$ such that 
$\t_Z(\sba)=s$. Then, by \cite[Table 2 and Proposition 5.5 (a)]{bonnafe quasi}, $\sba^4=1$. 
Moreover, $\Mbo^* \simeq \Hb \times (\Gb\Lb_2(\FM))^r$, where 
$\Hb\simeq\Sb\pb_{2m}(\FM)$ if $\diamondsuit=B$ and 
$\Hb\simeq \Sb\Ob_{2m}(\FM)$ if $\diamondsuit=D$. 

Write $\sba=(\sba',t_1,\dots,t_r)$ where $\sba' \in \Hb$ 
and $t_i \in \Gb\Lb_2(\FM)$. Since $s$ is quasi-isolated in 
$\Gb$, the eigenvalues $(\x_i,\x_i')$ of $t_i$, which are fourth 
roots of unity, satisfy $\x_i \x_i' \in \{1,-1\}$. 

Now, $F^*$ permutes the $r$ factors $\Gb\Lb_2(\FM)$. Assume that 
$F^*$ has an orbit of length greater than or equal to $3$. 
Then $S$ contains an element of order greater than or equal to 
$3^3-1$, which is impossible. Now let $\Sbo$ denote the center 
of a product of $\Gb\Lb_2(\FM)$ factors which are in the same 
orbit (we denote by $l$ the length of this orbit: we have $l\in \{1,2\}$). 
Let $\Sbt$ denote the torus of $\Gbt^*$ such that $\pi_Z(\Sbt)=\Sbo$ 
and let $\Sb=\t_Z(\Sbo)$. 
Since $l \le 2$ and $n \not= 8$, the map $\pi : \Sbt \to \Sb$ 
is an isomorphism of groups. Note that $\Sb \subset \Zb(\Mb)^\circ$, so 
$\Sbt^{F^*} \simeq \pi(\Sbt^{F^*}) \subset S'$. 
But, if $l = 2$, then $\Sbt^{F^*}$ contains an element of order greater 
than or equal to $8$ so $S'$ contains an element of order greater 
than or equal to $8$. This contradicts (P6). 
If $l=1$, let $z$ denote an element of $\Sbt^{F^*}$ of order $4$. 
Let $i \in \{1,2,\dots,r\}$ denote the place of the $\Gb\Lb_2(\FM)$ factor 
we are considering. Write 
$\sba\pi_Z(z)=(\sba',t_1,\dots,t_{i-1},t_i',t_{i+1},\dots,t_r)$ and let 
$\z_i$ and $\z_i'$ denote the eigenvalues of $t_i'$. 
Since $\x_i\x_i' \in \{1,-1\}$, we have $\{\x_i,\x_i'\}\not= \{\z_i,\z_i'\}$ 
so $s$ and $s\pi(z)$ are not conjugate in $\Gb^*$. This contradicts again (P6). 
\end{proof}

\bigskip

Before going on our investigation of the remaining cases (types $\Erm_6$, $\Erm_7$ and $\Erm_8$), 
we introduce the following property of the quadruple $(\Gb_0,\Mb_0,Z,n)$, 
where $Z$ is a subgroup of $\Zb(\Gb_0) \cap \Zb(\Mb_0)^\circ$ and $n$ is a non-zero
natural number.

\medskip

\begin{itemize}\itemindent1cm
\itemth{\SC_{\Gb_0,\Mb_0,Z,n}} {\it If $s$ is a semisimple element of $(\Mb_0/Z)^*$ which is 
quasi-isolated in $(\Mb_0/Z)^*$ and in $(\Gb_0/Z)^*$, then there exists an element 
$z \in \Zb((\Mb_0/Z)^*)^\circ$ of order dividing $n$ such that $s$ and $sz$ are 
not conjugate in $(\Gb_0/Z)^*$.}
\end{itemize}

\medskip

\noindent The property $(\SC_{\Gb_0,\Mb_0,Z,n})$ does not always hold (for instance, 
$\SC_{\Gb_0,\Mb_0,Z,1}$ never holds) but it can be tested 
with an algorithm using the \CHEVIE~package: this will be explained in the appendix. 
In the appendix, we will also present some examples for which property 
$(\SC_{\Gb_0,\Mb_0,Z,n})$ holds (see Lemma \ref{calcul}) and that will be used in the proof of the next 
lemmas (we will also give non-trivial examples in which $(\SC_{\Gb_0,\Mb_0,Z,n})$ does not hold).

\bigskip

\noindent{\bf Lemma E6.} 
{\it If $\Gb_0$ is of type $\Erm_6$, then $q=2$, $d=1$, $\Mb=\Mb_0$ is of type $\Arm_2 \times \Arm_2$ 
and $(\Gb,F)$ is not split.}

\bigskip

\begin{proof}[Proof of Lemma E6]
We assume in this subsection, and only in this subsection, that $\Gb_0$ is of type $\Erm_6$. Then, 
by Lemma \ref{fourre-tout} (a) and by \cite[\SEC 15.1]{luicc}, this implies that we are in one of the 
following cases:
\begin{quotation}
\begin{itemize}
\item[$\bullet$] $\Mb_0$ is a maximal torus.

\item[$\bullet$] $\Mb_0$ is of type $\Drm_4$ and $p=2$.

\item[$\bullet$] $\Mb_0$ is of type $\Arm_2 \times \Arm_2$, $p \neq 3$

\item[$\bullet$] $\Mb_0=\Gb_0$.
\end{itemize}
\end{quotation}
By (P3), the first and the last cases are excluded. By Lemma \ref{fourre-tout} (f), 
the second case is excluded. So $\Mb_0$ is of type $\Arm_2 \times \Arm_2$ and $p \neq 3$. 
Note that $\ZC(\Mb_0) \simeq \ZM/3\ZM \simeq \ZC(\Gb_0)$ (see \cite[Table 2.17]{bonnafe reg}) 
and that the graph automorphism of $\Gb_0$ acts non trivially on $\ZC(\Gb_0)$. So 
$e=1$ (by Lemma \ref{fourre-tout} (c)) and, if we denote by $\e \in \{1,-1\}$ the 
element defined by the condition $\e=1$ if and only if the graph automorphism 
of $\Gb_0$ induced by $F^d$ is trivial, then $F^d$ acts on $\ZC(\Mb_0) \simeq \ZM/3\ZM$ by 
multiplication by $\e q^d$. But, if we denote by $\chi$ the linear character of 
$\ZC(\Mb_0)$ associated with the cuspidal local system on $\Mb_0$, then $\chi$ 
is faithful \cite[\SEC 15.1]{luicc} and $F^d$-stable, so $\e q^d \equiv 1 \mod 3$. We can summarize 
these facts in the following statement:
\begin{quotation}
\begin{itemize}
\itemth{E_6^{(1)}} 
{\it $\Mb_0$ is of type $\Arm_2 \times \Arm_2$, $p \neq 3$, $e=1$ and $\e q^d \equiv 1 \mod 3$.}
\end{itemize}
\end{quotation}
In particular, it follows from (P6) that 
\begin{quotation}
\begin{itemize}
\itemth{E_6^{(2)}} {\it $s_0$ is quasi-isolated in $\Gb_0^*$ and $\Mb_0^*$ and 
$s_0$ is $\Gb_0^*$-conjugate to $s_0z$ for all $z \in \Zb(\Mb_0^*)^{\circ F^{*d}}$.} 
\end{itemize}
\end{quotation}
This implies that\begin{quotation}
\begin{itemize}
\itemth{E_6^{(3)}} {\it The order of $s_0$ divides $6$.}

\itemth{E_6^{(4)}} {\it All elements of $\Zb(\Mb_0^*)^{\circ F^{*d}}$ have order dividing $6$.}

\itemth{E_6^{(5)}} {\it $\Zb(\Mb_0^*)^{\circ F^{*d}}$ does not contain all elements of order $3$ 
of $\Zb(\Mb_0^*)^\circ$.}
\end{itemize}
\end{quotation}
Indeed, $(\Erm_6^{(3)})$ follows from \cite[Table 3]{bonnafe quasi}, $(\Erm_6^{(4)})$ 
follows from $(\Erm_6^{(3)})$ and Lemma \ref{fourre-tout} (e), 
while $(\Erm_6^{(5)})$ follows from the 
fact that $(\SC_{\Gb_0,\Mb_0,1,3})$ holds (see Lemma \ref{calcul} (1)), 
which is proved by computer calculation in the Appendix A.

Now, by Lemma \ref{fourre-tout} (d) and $(\Erm_6^{(4)})$, we get that $q^d \le 7$ and $q \le 5$. 
Recall also that $p \neq 3$ by $(\Erm_6^{(1)})$. So
\begin{quotation} 
\begin{itemize}
\itemth{E_6^{(6)}} {\it The pair $(q,d)$ belongs to $\{(2,1), (4,1), (5,1), (2,2)\}$.} 
\end{itemize}
\end{quotation}
Moreover, $\dim \Zb(\Mb_0^*)^\circ = 2$. 
If $q^d=4$, then Lemma \ref{fourre-tout} (h) and $(\Erm_6^{(4)})$ force that 
$\Zb(\Mb_0^*)^{\circ F^{*d}} \simeq \ZM/3\ZM \times \ZM/3\ZM$, so 
$\Zb(\Mb_0^*)^{\circ F^{*d}}$ contains all elements of order $3$ of $\Zb(\Mb_0^*)^\circ$. 
This contradicts $(\Erm_6^{(5)})$. 

Similarly, if $q^d=5$, then Lemma \ref{fourre-tout} (h) and $(\Erm_6^{(4)})$ forces that 
$\Zb(\Mb_0^*)^{\circ F^{*d}} \simeq \ZM/6\ZM \times \ZM/6\ZM$, so 
$\Zb(\Mb_0^*)^{\circ F^{*d}}$ contains all elements of order $3$ of $\Zb(\Mb_0^*)^\circ$. 
Again, this contradicts $(\Erm_6^{(5)})$. Therefore $q=2$ and $d=1$: 
in particular, $\e=-1$ by $(\Erm_6^{(1)})$. 
\begin{quotation}
\begin{itemize}
\itemth{E_6^{(7)}} {\it $q=2$, $d=1$ and the pair $(\Gb,F)$ is not split}.
\end{itemize}
\end{quotation}
So Lemma E6 follows from $(\Erm_6^{(1)})$ and $(\Erm_6^{(7)})$.
\end{proof}

\bigskip

\noindent{\bf Lemma E7.} 
{\it The group $\Gb_0$ is not of type $\Erm_7$.}

\bigskip

\begin{proof}[Proof of Lemma E7]
We assume in this subsection, and only in this subsection, 
that $\Gb_0$ is of type $\Erm_7$. Then, by (4), 
the group $\Mb_0$ admits an $F^d$-stable cuspidal local system supported 
by a unipotent class. By \cite[\SEC 15.1]{luicc}, this implies that we are in one of the 
following cases:
\begin{quotation}
\begin{itemize}
\item[$\bullet$] $\Mb_0$ is a maximal torus.

\item[$\bullet$] $\Mb_0$ is of type $\Arm_1 \times \Arm_1 \times \Arm_1$, 
$\ker h_{\Mb_0}^{\Gb_0}=1$, and $p \neq 2$.

\item[$\bullet$] $\Mb_0$ is of type $\Drm_4$ and $p=2$.

\item[$\bullet$] $\Mb_0$ is of type $\Erm_6$ and $p=3$.

\item[$\bullet$] $\Mb_0=\Gb_0$.
\end{itemize}
\end{quotation}
By (P3), the first and the last cases are excluded. By Lemma \ref{fourre-tout} (f), 
the third case is excluded. We will now investigate the two remaining cases.

\medskip

$\bullet$ Assume first that $\Mb_0$ is of type $\Erm_6$ and that $p=3$. 
By \cite[Table 2.17]{bonnafe reg}, we get that $\ZC(\Mb_0)=1$ so 
$\Zb(\Gb_0) \subseteq \Zb(\Mb_0)^\circ$. Now, $(\Gb_0/\Zb(\Gb_0))^*=\Gbt_0^*$, 
so $s_{\Zb(\Gb_0)}$ is isolated in $\Gbt_0^*$: in particular, its order 
belongs to $\{1,2,3,4\}$. So it follows from Lemma \ref{fourre-tout} (c) and (d) 
that $q^d \le 3$. So $q=3$ and $d=1$. Since $\dim \Zb(\Mbt^*)^\circ = 1$ 
and since $\Zb(\Mbt^*)^\circ$ is not split by Lemma \ref{fourre-tout} (b) and (g), 
this forces $S_{\Zb(\Gb_0)}  \simeq \ZM/4 \ZM$. So $S_{\Zb(\Gb_0)}'=S_{\Zb(\Gb_0)}$ 
contains all the elements of $\Zb(\Mbt^*)^\circ$ of order dividing $4$. 
But then Lemma \ref{calcul} (2) contradicts (P6). 
So this case cannot occur.

\bigskip

$\bullet$ So assume now that $\Mb_0$ is of type $\Arm_1 \times \Arm_1 \times \Arm_1$, 
that $\Ker h_{\Mb_0}^{\Gb_0}=1$ and that $p \neq 2$. Note that these conditions 
describe completely the type of the pair $(\Gb_0,\Mb_0)$. Indeed, 
it is given by the following diagram, where the three black nodes correspond 
to the simple roots of $\Gb_0$ which are simple roots of $\Mb_0$:
\begin{center}
\begin{picture}(240,45)
\put( 40, 30){\circle{10}}
\put( 45, 30){\line(1,0){30}}
\put( 80, 30){\circle{10}}
\put( 85, 30){\line(1,0){30}}
\put(120, 30){\circle{10}}
\put(125, 30){\line(1,0){30}}
\put(160, 30){\circle*{10}}
\put(165, 30){\line(1,0){30}}
\put(200, 30){\circle{10}}
\put(205, 30){\line(1,0){30}}
\put(240, 30){\circle*{10}}
\put(120, 5){\circle*{10}}
\put(120, 25){\line(0,-1){15}}
\put(-93,15){$(\Erm_7[\Arm_1^3]^\#)$}
\end{picture}
\end{center}
By \cite[\SEC 15.2]{luicc} and \cite[Table 2.17]{bonnafe reg}, 
we get that $e=1$. By \cite[Table 3]{bonnafe quasi}, $o(s) \in \{1,2,3,4,6\}$. 
So it follows from Lemma \ref{fourre-tout} (c) and (d) that $q^d \le 7$ and $q \le 5$. 
Moreover, since $p > 2$, we get:
\begin{quotation}
\begin{itemize}
\itemth{E_7^{(1)}} {\it $d=e=1$ and $q \in \{3,5\}$}.
\end{itemize}
\end{quotation}
Moreover, by Lemma \ref{calcul} (3), we get that $(\SC_{\Gb_0,\Mb_0,1,4})$ and 
$(\SC_{\Gb_0,\Mb_0,1,6})$ hold, so 
\begin{quotation}
\begin{itemize}
\itemth{E_7^{(2)}} {\it $\Zb(\Mb_0^*)^{\circ F^*}$ does not contain all elements of order $4$ 
of $\Zb(\Mb_0^*)^\circ$.}

\itemth{E_7^{(2')}} {\it $\Zb(\Mb_0^*)^{\circ F^*}$ does not contain all elements of order $6$ 
of $\Zb(\Mb_0^*)^\circ$.}
\end{itemize}
\end{quotation}
Since $o(s) \in \{1,2,3,4,6\}$, we get that 
\begin{quotation}
\begin{itemize}
\itemth{E_7^{(3)}} {\it Every element of $\Zb(\Mb_0^*)^{\circ F^*}$ has order in $\{1,2,3,4,6\}$.}
\end{itemize}
\end{quotation}
For simplification, let $\chi=\chi_{\Zb(\Mb_0^*)^\circ,F^*}$. 
If $\chi$ contains some factor $\Phi_m$ with $m \ge 3$, then it follows 
from Lemma \ref{ordre tore} that $S_Z$ contains an element of order 
$\Phi_m(q) \ge \Phi_m(3) \ge \Phi_6(3) = 7$, so this contradicts $(\Erm_7^{(3)})$. 
Since $\dim \Zb(\Mb_0^*) = 4$, 
$\chi=\Phi_1^a \Phi_2^b$, with $a+b = 4$. If $a$, $b \ge 1$, then 
it follows from Lemma \ref{ordre 8} of the Appendix that 
$\Zb(\Mb_0^*)^{\circ F^*}$ contains an element of order $\ge 8$, which is impossible 
by $(\Erm_7^{(3)})$. Moreover $b \neq 0$ by Lemma \ref{fourre-tout} (b). So 
$\chi=\Phi_2^4$. But then $(\Zb(\Mb_0^*)^\circ,F^*)$ is a $\Phi_2$-torus. It then follows from 
\ref{ordre tore} that 
$$S_Z \simeq \bigl(\ZM/(q+1)\ZM\bigr)^4.$$
Since $q \in \{3,5\}$, this contradicts $(\Erm_7^{(2)})$ and $(\Erm_7^{(2')})$.
\end{proof}

\bigskip

\noindent{\bf Lemma E8.} 
{\it The group $\Gb_0$ is not of type $\Erm_8$.}

\bigskip

\begin{proof}[Proof of Lemma E8]
First, note that $\Gb_0$ is simply-connected and adjoint, so $e=1$. It then follows from (P6) that 
\begin{quotation}
\begin{itemize}
\itemth{E_8^{(1)}} $1 \le o(s_0) \le 6$.

\itemth{E_8^{(2)}} {\it All the elements of $\Zb(\Mb_0^*)^{\circ F^{*d}}$ have order 
in $\{1,2,3,4,5,6\}$.}
\end{itemize}
\end{quotation}
Indeed, $(\Erm_8^{(1)})$ follows from \cite[Proposition 4.9]{bonnafe quasi}, 
$(\Erm_8^{(2)})$ follows from $(\Erm_8^{(1)})$ and (P6).
Moreover, by Lemma \ref{fourre-tout} (a) and \cite[\SEC 15.3]{luicc}, 
we are in one of the following cases:
\begin{quotation}
\begin{itemize}
\item[$\bullet$] $\Mb_0$ is a maximal torus.

\item[$\bullet$] $\Mb_0$ is of type $\Drm_4$ and $p=2$.

\item[$\bullet$] $\Mb_0$ is of type $\Erm_6$ and $p=3$.

\item[$\bullet$] $\Mb_0$ is of type $\Erm_7$ and $p=2$.

\item[$\bullet$] $\Mb_0=\Gb_0$.
\end{itemize}
\end{quotation}
By (P3), the first and the last cases are excluded. By Lemma \ref{fourre-tout} (f), 
the second case is excluded. We will investigate the two remaining cases.

\medskip

$\bullet$ Assume first that $\Mb_0$ is of type $\Erm_6$ and that $p=3$. 
Then, since $(\SC_{\Gb_0,\Mb_0,1,2})$ holds by Lemma \ref{calcul} (5), it follows that 
\begin{quotation}
\begin{itemize}
\itemth{E_8^{(3)}} {\it $\Zb(\Mb_0^*)^{\circ F^{*d}}$ does not contain all elements 
of order $2$ of $\Zb(\Mb_0^*)^\circ$.}
\end{itemize}
\end{quotation}
By Lemma \ref{fourre-tout} (c), $(\Erm_8^{(2)})$ forces $q^d \le 7$. Since $p=3$, we get that 
$q=3$ and $d=1$. Moreover, since $\dim \Zb(\Mb^*)^\circ=2$, we get that 
$\Zb(\Mb^*)^{\circ F^*}$ is isomorphic to one of the following groups
$$\ZM/2\ZM \times \ZM/2\ZM,\quad \ZM/2\ZM \times \ZM/4\ZM,\quad 
\ZM/4\ZM \times \ZM/4\ZM,$$
$$\ZM/(3^2-1)\ZM=\ZM/8\ZM,\quad \ZM/\Phi_3(3)\ZM=\ZM/13\ZM\quad
\text{or}\quad \ZM/\Phi_6(3)\ZM=\ZM/7\ZM.$$
But this contradicts the fact that $(\Erm_8^{(2)})$ and $(\Erm_8^{(3)})$ both hold.

\medskip

$\bullet$ Assume now that $\Mb_0$ is of type $\Erm_7$ and that $p=2$. 
It then follows from Lemma \ref{calcul} (4) that
\begin{quotation}
\begin{itemize}
\itemth{E_8^{(4)}} {\it $\Zb(\Mb_0^*)^{\circ F^{*d}}$ does not contain all elements 
of order $3$ of $\Zb(\Mb_0^*)^\circ$.}
\itemth{E_8^{(4')}} {\it $\Zb(\Mb_0^*)^{\circ F^{*d}}$ does not contain all elements 
of order $5$ of $\Zb(\Mb_0^*)^\circ$.}
\end{itemize}
\end{quotation}
By Lemma \ref{fourre-tout} (c), $(\Erm_8^{(2)})$ forces $q^d \le 7$. Since $p=2$, 
this implies that $q^d \in \{2,4\}$. But $\dim \Zb(\Mb_0^*)^\circ=1$, so 
$\Zb(\Mb_0^*)^{\circ F^{*d}}$ is isomorphic to one of the following groups 
$$\ZM/(q^d-1)\ZM\quad\text{or}\quad \ZM/(q^d+1)\ZM.$$
In other words, $\Zb(\Mb_0^*)^{\circ F^{*d}}$ is isomorphic to $\ZM/3\ZM$ or $\ZM/5\ZM$. 
This contradicts $(\Erm_8^{(4)})$ or $(\Erm_8^{(4')})$.
\end{proof}

\bigskip

Now, the proof of the Proposition \ref{strange} is complete by (P2) and the 
Lemmas C, BD, FG, E6, E7 and E8.

\bigskip

\section{Application to the Mackey formula for Lusztig induction and restriction}

\medskip

This section is devoted to the proof of the main result of this paper, 
namely the Theorem stated in the introduction. We shall fix some notation: 
if $\Lb$ and $\Mb$ are respective $F$-stable Levi complements of parabolic subgroups 
$\Pb$ and $\Qb$ of $\Gb$, we set
$$\D_{\Lb \incl \Pb, \Mb \incl \Qb}^\Gb=
\lexp{*}{R}_{\Lb \incl \Pb}^\Gb \ci R_{\Mb \incl \Qb}^\Gb - \!\!
\sum_{g \in \Lb^F\backslash\SC_\Gb(\Lb,\Mb)^F/\Mb^F} \!\!
R_{\Lb \cap \lexp{g}{\Mb} \incl \Lb \cap \lexp{g}{\Qb}}^\Lb \ci 
\lexp{*}{R}_{\Lb \cap \lexp{g}{\Mb} \incl \Pb \cap \lexp{g}{\Mb}}^{\lexp{g}{\Mb}}
\ci (\ad g)_\Mb.$$
The Mackey formula $(\MC_{\Gb,\Lb,\Pb,\Mb,\Qb})$ is then equivalent to the 
equality
$$\D_{\Lb \subset \Pb, \Mb \subset \Qb}^\Gb = 0.\leqno{(\MC_{\Gb,\Lb,\Pb,\Mb,\Qb})}$$

\bigskip

\subsection{Preliminaries}
First, note that
\equat\label{dsg}
\text{\it $f=0$ if and only if $d_s^\Gb f = 0$ for all semisimple elements $s \in \Gb^F$.}
\endequat
We now recall some results from \cite{bonnafe a}: these are some properties of the maps 
$\D_{\Lb \subset \Pb,\Mb \subset \Qb}^\Gb$ 
which can be proved a priori (see \cite{bonnafe a}). 

\def\espace{\vphantom{\frac{\DS{A^A_A}}{\DS{A_A^A}}}}

First of all
\equat\label{delta adjoint}
\text{\it $\D_{\Lb \subset \Pb,\Mb \subset \Qb}^\Gb$ and $\D_{\Mb \subset \Qb,\Lb \subset \Pb}^\Gb$ 
are adjoint for the scalar products $\langle,\rangle_{\Lb^F}$ and $\langle,\rangle_{\Mb^F}$.}
\endequat
Let $\Pb'$ and $\Qb'$ be parabolic subgroups of $\Gb$ and let $\Lb'$ and $\Mb'$ 
be the unique Levi complement of $\Pb'$ and $\Qb'$: we assume that $\Lb'$ and $\Mb'$ are $F$-stable 
and that $\Lb \subset \Lb'$, $\Pb \subset \Pb'$, $\Mb \subset \Mb'$ and $\Qb \subset \Qb'$. Then 
\cite[Proposition 1 (c) of the Corrigenda]{bonnafe a}
\equat\label{delta cuspidal}
\begin{array}{rcl}
\espace \D_{\Lb \subset \Pb,\Mb \subset \Qb}^\Gb &=& \lexp{*}{R}_{\Lb \subset \Pb \cap \Lb'}^{\Lb'} 
\circ \D_{\Lb' \subset \Pb',\Mb' \subset \Qb'}^\Gb \circ R_{\Mb \subset \Qb \cap \Mb'}^{\Mb'} \\
&&\espace + \DS{\sum_{g \in \Lb^{\prime F}\backslash \SC_\Gb(\Lb',\Mb')/\Mb^{\prime F}}} 
\lexp{*}{R}_{\Lb \subset \Pb \cap \Lb'}^{\Lb'} \circ 
R_{\Lb' \cap \lexp{g}{\Mb'} \subset \Lb' \cap \lexp{g}{\Qb'}}^{\Lb'} \\
&&\espace \circ \D_{\Lb' \cap \lexp{g}{\Mb'} \subset \Pb' \cap \lexp{g}{\Mb'},\lexp{g}{\Mb} 
\subset \lexp{g}{(\Qb \cap \Mb')}}^{\lexp{g}{\Mb'}}\circ 
(\ad g)_\Mb \\
&&\espace + \DS{\sum_{g \in \Lb^{\prime F}\backslash \SC_\Gb(\Lb',\Mb)^F/\Mb^F} }
\D_{\Lb \subset \Pb \cap \Lb',\Lb' \cap \lexp{g}{\Mb} \subset \Lb' \cap \lexp{g}{\Qb}}^{\Lb'} \\
&&\espace \circ R_{\Lb' \cap \lexp{g}{\Mb} \subset \Pb' \cap \lexp{g}{\Mb}}^{\lexp{g}{\Mb}} \circ 
(\ad g)_\Mb.
\end{array}
\endequat
Also, if $s \in \Lb^F$ is semisimple, then \cite[5.1.5]{bonnafe a}
\equat\label{dsg delta}
d_s^\Lb \circ \D_{\Lb \subset \Pb,\Mb \subset \Qb}^\Gb = \!\!
\sum_{\substack{g \in \Gb^F \\ \SSS{\text{such that }} s \in \lexp{g}{\Mb}}} \!\!
\frac{|C_{\lexp{g}{\Mb}}^\circ(s)^F|}{|\Mb^F|\cdot|C_\Gb^\circ(s)^F|} 
\D_{C_\Lb^\circ(s) \subset C_\Pb^\circ(s),
C_{\lexp{g}{\Mb}}^\circ \subset C_{\lexp{g}{\Qb}}^\circ(s)}^{C_\Gb^\circ(s)} \circ 
d_s^{\lexp{g}{\Mb}} \circ (\ad g)_\Mb.
\endequat
In particular, if $z \in \Zb(\Gb)^F$, then \cite[5.1.6]{bonnafe a}
\equat\label{dzg delta}
d_z^\Lb \circ \D_{\Lb \subset \Pb,\Mb \subset \Qb}^\Gb = 
\D_{\Lb \subset \Pb,\Mb \subset \Qb}^\Gb \circ d_z^\Mb.
\endequat
Finally, if $\Gbh$ denotes a connected reductive group endowed with an 
$\gfq$-Frobenius endomorphism (still denoted by $F$) and if $i : \Gbh \to \Gb$ 
is a morphism of algebraic groups defined over $\gfq$ and such that $\Ker i$ is central in $\Gbh$ and 
$\Im i$ contains the derived subgroup of $\Gb$, then \cite[Proposition 1.1]{bonnafe rouquier}
\equat\label{res delta}
\Res_{\Lbh^F}^{\Lb^F} \circ \D_{\Lb \subset \Pb,\Mb \subset \Qb}^\Gb 
= \D_{\Lbh \subset \Pbh,\Mbh \subset \Qbh}^\Gb \circ \Res_{\Mbh^F}^{\Mb^F}. 
\endequat
Here, $\hat{?}=i^{-1}(?)$ for $? \in \{\Lb,\Pb,\Mb,\Qb\}$ and 
$\Res_{\Lbh^F}^{\Lb^F} : \Class(\Lb^F) \to \Class(\Lbh^F)$, $f \mapsto f \circ i$. 
In this last situation, we shall need the following lemma:

\bigskip

\begin{lem}\label{i G}
If $\Ker i \subseteq \Zb(\Gbh)^\circ$ and if $u$ and $v$ are two {\bfit unipotent} elements 
of $\Gbh^F$, then $u$ and $v$ are conjugate in $\Gbh^F$ if and only if $i(u)$ and $i(v)$ 
are conjugate in $\Gb^F$.
\end{lem}

\bigskip

\begin{proof}
Of course, if $u$ and $v$ are conjugate in $\Gbh^F$, then $i(u)$ and $i(v)$ 
are conjugate in $\Gb^F$. Conversely, assume that $i(u)$ and $i(v)$ 
are conjugate in $\Gb^F$. Then there exists $g \in \Gbh$ such that 
$i(g) \in \Gb^F$ and $i(gug^{-1})=i(v)$. So there exists $z \in \Ker i$ such that 
$gug^{-1}=zv$. Since $u$ and $v$ are unipotent, this forces $z=1$. 
So $gug^{-1}=v$.

On the other hand, as $i(g) \in \Gb^F$, we get that $g^{-1}F(g) \in \Ker i \subseteq \Zb(\Gbh)^\circ$. 
By Lang's Theorem, there exists $z_\circ \in \Zb(\Gbh)^\circ$ such that 
$z_\circ F(z_\circ^{-1})=g^{-1}F(g)$. Then $z_\circ g \in \Gbh^F$ and 
$(z_\circ g) u (z_\circ g)^{-1} = v$.
\end{proof}

\bigskip

\subsection{Main Theorem}
We are now ready to prove the Theorem announced in the introduction:

\bigskip

\begin{theo}\label{mackey}
Assume that one of the following holds:
\begin{itemize}
\itemth{1} $\Pb$ and $\Qb$ are $F$-stable (Deligne \cite[Theorem 2.5]{luspa}).

\itemth{2} $\Lb$ or $\Mb$ is a maximal torus of $\Gb$ (Deligne and 
Lusztig \cite[Theorem 7]{delu}).

\itemth{3} $q > 2$.

\itemth{4} $\Gb$ does not contain an $F$-stable quasi-simple component 
of type $\lexp{2}{\Erm_6}$, $\Erm_7$ or $\Erm_8$.
\end{itemize}
Then the Mackey formula $\D_{\Lb \subset \Pb, \Mb \subset \Qb}^\Gb = 0$ holds.
\end{theo}

\bigskip

\begin{proof}
For simplification, we will denote by (P0) the following assertion on $\Gb$:
\begin{itemize}
\itemth{P0} {\it $q > 2$ or $\Gb$ does not contain an $F$-stable quasi-simple component 
of type $\lexp{2}{\Erm_6}$, $\Erm_7$ or $\Erm_8$.}
\end{itemize}
In other words, (P0) is equivalent to say that $\Gb$ satisfies at least one of the assertions 
(3) or (4) of the Theorem \ref{mackey}. 

\medskip

Our proof of Theorem \ref{mackey} follows an induction argument. 
We denote by $\chi(\Gb)$ the order of the torsion group of $Y(\Tb)/\langle \Phi^\vee \rangle$, 
where $\Tb$ is a maximal torus of $\Gb$ and $\Phi^\vee \subset Y(\Tb)$ is its coroot system. 
We set 
$$\nb_{\Gb,\Lb,\Mb}=(\dim \Gb , \dim \Lb + \dim \Mb,\chi(\Gb)) \in \NM \times \NM \times \NM.$$
We shall denote by $\infspe$ the lexicographic order on $\NM \times \NM \times \NM$ and we 
assume that we have found a sextuple $(\Gb,F,\Lb,\Pb,\Mb,\Qb)$ which satisfies 
(P0) and such that $\D_{\Lb \subset \Pb, \Mb \subset \Qb}^\Gb \neq 0$ 
with $\nb_{\Gb,\Lb,\Mb}$ is minimal (for the lexicographic order $\infspe$). 
Our aim is to show that $(\Gb,F,\Lb,\Pb)$ 
or $(\Gb,F,\Mb,\Qb)$ satisfies all the properties (Pk), $1 \le k \le 6$. 
Then we get a contradiction, since Proposition \ref{strange} shows 
that there is no quadruple $(\Gb,F,\Mb,\Qb)$ satisfying 
(P0), (P1), (P2), (P3), (P4), (P5) and (P6) together.

\medskip

For this purpose, we shall need the following trivial remark, that will allow to 
use an induction argument: 
\begin{quotation}
\begin{itemize}
\itemth{IND} {\it If $\Gb$ satisfies $({\mathrm{P0}})$ and if $\Hb$ is a connected reductive 
subgroup of $\Gb$ of the same rank, then $\Hb$ satisfies also $({\mathrm{P0}})$.}
\end{itemize}
\end{quotation}

\medskip

$\bullet$ {\it First step: proof of $({\mathrm{P5}})$.} 
Assume that $\Pb$ and $\Qb$ are contained in proper $F$-stable parabolic subgroups 
$\Pb'$ and $\Qb'$ respectively. Let $\Lb'$ (respectively $\Mb'$) be the 
unique Levi complement of $\Pb'$ (respectively $\Qb'$) containing $\Lb$ 
(respectively $\Mb$). Then $\Lb'$ and $\Mb'$ are $F$-stable by uniqueness 
and $\D_{\Lb' \subset \Pb', \Mb' \subset \Qb'}^\Gb=0$ by Theorem \ref{mackey} (1). 
So it follows from the minimality of $\nb_{\Gb,\Lb,\Mb}$, from \ref{delta cuspidal} 
and from (IND) that $\D_{\Lb \subset \Pb,\Mb \subset \Qb}^\Gb = 0$, 
contrarily to our hypothesis. 

Therefore, $\Pb$ or $\Qb$ is not contained in a proper $F$-stable parabolic subgroup 
of $\Gb$. By \ref{delta adjoint}, we have also that $\D_{\Mb \subset \Qb,\Lb \subset \Pb}^\Gb \neq 0$, 
so we may assume that $\Qb$ is not contained in a proper $F$-stable parabolic subgroup 
of $\Gb$. 

Therefore, from now on, we will prove that $(\Gb,F,\Mb,\Qb)$ satisfies 
all of the properties (Pk), $1 \le k \le 6$. We have just proved (P5).

\medskip

$\bullet$ {\it Second step: proof of $({\mathrm{P3}})$.} 
This follows immediately from Theorem \ref{mackey} (2).

\medskip

$\bullet$ {\it Third step: proof of $({\mathrm{P1}})$.} 
Let $\mu \in \Class(\Gb^F)$ be such that $\D_{\Lb \subset \Pb,\Mb \subset \Qb}^\Gb (\mu) \neq 0$ and 
let $s \in \Lb^F$ be semisimple. 
By the minimality of $\dim \Gb$ and by (IND), it follows from 
\ref{dsg delta} that $d_s^\Gb \circ \D_{\Lb \subset \Pb,\Mb \subset \Qb}^\Gb=0$ 
if $s \not\in \Zb(\Gb)^F$. So, by \ref{dsg}, there exists $z \in \Zb(\Gb)^F$ such that 
$d_z^\Lb (\D_{\Lb \subset \Pb,\Mb \subset \Qb}^\Gb (\mu) ) \neq 0$. 
Therefore, $\D_{\Lb \subset \Pb,\Mb \subset \Qb}^\Gb(d_z^\Mb \mu) = 0$ by \ref{dzg delta}. 
In other words, if we replace $\mu$ by $d_z^\Mb \mu$, 
this means that we may (and we will) assume that $\mu$ has support 
on unipotent elements of $\Mb^F$. 

Now, let $i : \Gbh \to \Gb$ be the simply-connected covering of the derived subgroup 
of $\Gb$ and let $F : \Gbh \to \Gbh$ denote the unique $\gfq$-Frobenius endomorphism 
such that $i$ is defined over $\gfq$. Let $\hat{?}=i^{-1}(?)$, 
for $? \in \{\Lb,\Pb,\Mb,\Qb\}$. Since $\mu$ has unipotent support, 
$\D_{\Lb \subset \Pb,\Mb \subset \Qb}^\Gb (\mu)$ has also a unipotent support. 
Moreover, since $i$ induces a bijective morphism between the unipotent varieties, 
it induces a bijection between unipotent elements of $\Gbh^F$ and unipotent elements 
of $\Gb^F$. Therefore, $\Res_{\Lbh^F}^{\Lb^F} \D_{\Lb \subset \Pb,\Mb \subset \Qb}^\Gb (\mu) \neq 0$. 
By \ref{res delta}, this means that $\D_{\Lbh \subset \Pbh,\Mbh \subset \Qbh}^\Gbh \neq 0$. 
But $\nb_{\Gbh,\Lbh,\Mbh} \infspe \nb_{\Gb,\Lb,\Mb}$, with equality if and only if 
$i$ is an isomorphism. By the minimality of $\nb_{\Gb,\Lb,\Mb}$, we get that $i$ 
is an isomorphism. So $\Gb$ is semisimple and simply-connected. 

By writing $\Gb$ as the product of its quasi-simple components, one can write 
$\Gb$ as a direct product of semisimple simply-connected $F$-stable 
groups $\Gb_j$, $j \in J$ ($J$ being some indexing set), the Frobenius endomorphism acting transitively 
on the quasi-simple components of $\Gb_j$. Since Lusztig functors are compatible 
with direct products, the minimality of $\nb_{\Gb,\Lb,\Mb}$ implies 
that $\Gb$ is one of these $\Gb_j$'s. Therefore, 
$\Gb$ is semisimple and simply-connected, and $F$ permutes transitively 
the quasi-simple components of $\Gb$. This completes the proof 
of (P1).

\medskip

$\bullet$ {\it Fourth step: proof of $({\mathrm{P2}})$.} 
This follows immediately from \cite[Theorem 5.2.1]{bonnafe a}.

\medskip

$\bullet$ {\it Fifth step: proof of $({\mathrm{P4}})$.} 
Recall that we have found a class function $\mu$ on $\Mb^F$, with unipotent support, 
and such that $\D_{\Lb \subset \Pb,\Mb \subset \Qb}^\Gb (\mu) \neq 0$. 
Let $\Class_\uni(\Gb^F)$ denotes the subspace of $\Class(\Gb^F)$ consisting 
of functions with unipotent support. In other words, $\Class_\uni(\Gb^F)$ 
is the image of $d_1^\Gb$. 

Let $\EC_\Mb$ denote the subspace of $\Class_\uni(\Mb^F)$ generated by 
all the $R_{\Mb' \subset \Qb'}^\Mb(\m')$, where $\Mb'$ is an $F$-stable 
Levi complement of a {\it proper} parabolic subgroup $\Qb'$ of $\Mb$ 
and $\mu' \in \Class_\uni(\Mb^{\prime F})$. 
Then it follows from the minimality of $\nb_{\Gb,\Lb,\Mb}$, from (IND) and 
from \ref{delta cuspidal} (see also \cite[5.1.8 and Proposition 1 (a) of the Corrigenda]{bonnafe a} 
for a particular form of this formula) that 
$\D_{\Lb \subset\Pb, \Mb \subset \Qb}^\Gb(\EC_\Mb)=0$. 
So, if we write $\mu=\mu_c + \mu'$, with $\mu' \in \EC_\Mb$ and $\mu_c \in \EC_\Mb^\perp$, 
then $\D_{\Lb \subset\Pb, \Mb \subset \Qb}^\Gb(\mu)=
\D_{\Lb \subset\Pb, \Mb \subset \Qb}^\Gb(\mu_c)\neq 0$. 
This means that the vector space $\EC_\Mb^\perp$ is non-zero and 
that we may assume that $\mu=\mu_c$. 
But $\EC_\Mb^\perp$ is the space of {\it absolutely cuspidal functions on $\Mb^F$ 
with unipotent support} (as defined in \cite[\SEC 3.1]{bonnafe a}: it was denoted 
by $\Cus_\uni(\Mb^F)$ in this paper). 

Now, by the minimality of $\nb_{\Gb,\Lb,\Mb}$, it follows that 
{\it the Mackey formula holds in $\Mb$} (in the sense of \cite[Definition 1.4.2]{bonnafe a}). 
So it follows from \cite[Corollary 8 of the Corrigenda]{bonnafe a} that 
there exists an $F$-stable unipotent class of $\Mb$ which supports a cuspidal 
local system. This shows (P4).

\medskip

$\bullet$ {\it Sixth step: proof of $({\mathrm{P6}})$.}   
Now, let $\Zb$ be an $F$-stable subgroup of $\Zb(\Mb)^\circ \cap \Zb(\Gb)$. 
Note that $\Zb$ is finite (since $\Gb$ is semisimple). Let $\Gbo=\Gb/\Zb$. 
If $? \in \{\Lb,\Pb,\Mb,\Qb\}$, we set $\overline{?} =? \cap \Gbo$. 
Note that $\dim \Gbo = \dim \Gb$ and $\dim \Lbo+\dim \Mbo=\dim \Lb + \dim \Mb$. 
Let $\t : \Gb \to \Gbo$ denote the canonical morphism. 
Let $u$ and $v$ be two unipotent elements of $\Mb^F$. Since 
$\Zb \incl \Zb(\Mb)^\circ$, $u$ and $v$ are conjugate in $\Mb^F$ if and only 
if $\t(u)$ and $\t(v)$ are conjugate in $\Mbo^F$ (see Lemma \ref{i G}).
So there exists a unique $\fba \in \Class(\Gbo^F)$ such that 
$f=d_1^\Gb \Res_{\Mb^F}^{\Mbo^F} \fba$. 
Moreover, since $\Zb \incl \Zb(\Mb)^\circ$ and by \ref{dzg delta} and \ref{res delta}, 
$$d_1^\Lb \Res_{\Lb^F}^{\Lbo^F} \circ \D_{\Lbo \incl \Pbo,\Mbo \incl \Qbo}^\Gbo (\fba)
= \D_{\Lb \incl \Pb,\Mb \incl \Qb}^\Gb(f).$$
So $\D_{\Lbo \incl \Pbo,\Mbo \incl \Qbo}^\Gbo (\fba) \not= 0$. 
In particular, there exists an irreducible character $\m$ of $\Mbo^F$ such 
that $\D_{\Lbo \incl \Pbo,\Mbo \incl \Qbo}^\Gbo (\mu) \not= 0$. 
Let $s \in \Mbo^{*F^*}$ be semisimple and such that 
$\mu \in \EC(\Mbo^F,[s]_{\Mbo^{*F^*}})$. By the argument 
in \cite[Lemmas 5.1.3 and 5.1.4]{bonnafe q}, 
and by the minimality of $(\dim \Gbo,\dim \Lbo + \dim \Mbo)=
(\dim \Gb,\dim \Lb + \dim \Mb) \in \NM \times \NM$ (where $\NM \times \NM$ is ordered 
lexicographically), $s$ is quasi-isolated in $\Mbo$ and in $\Gbo$. Moreover, by the argument 
at the end of the proof of \cite[Theorem 5.1.1]{bonnafe q}, 
$s$ is conjugate to $sz$ (in $\Gb^*$) 
for every $z \in \Zb(\Mbo^*)^{F^*} \cap \pi_Z^*(\Gbt^{*F^*})$, 
where $\pi_Z^* : \Gbt^* \to \Gbo^*$ is the simply connected covering of $\Gbo^*$. 
So we have proved (P6).
\end{proof}

\bigskip

\remark{2E6}
In  fact, our proof shows that, if we  are able to prove the Mackey formula
$(\MC_{\Gb,\Lb,\Pb,\Mb,\Qb})$   whenever   $(\Gb,F)$   is   semisimple  and
simply-connected  of type  $\lexp{2}{\Erm}_6$, $q=2$  and $\Mb$  is of type
$\Arm_2  \times \Arm_2$,  then the  Mackey formula  would hold for any pair
$(\Gb,F)$,  where $\Gb$ is  a connected reductive  algebraic group and $F :
\Gb \to \Gb$ is a Frobenius endomorphism.

Actually  it can be shown that the problem reduces to prove that the scalar
product $\langle R_\Mb^\Gb \Gamma_\zeta^\Mb, R_\Mb^\Gb
\Gamma_\zeta^\Mb\rangle$  has the  value predicted  by the  Mackey formula,
where $\zeta$ is a faithful character of $H^1(F,\Zb(\Mb))$ and
$\Gamma_\zeta^\Mb$ is the corresponding Mellin transform of a Gelfand-Graev
character  (see \cite[Theorem 6.2.1]{bonnafe a}).  We were unfortunately unable
to do this.\finl

\bigskip

\renewcommand\thesection{\Alph{section}}
\setcounter{section}{0}

\section{Appendix: computations with semisimple elements using \CHEVIE}

\medskip

In this Appendix, we present briefly some algorithms and some programs using 
the \CHEVIE~package for computing with semisimple elements in reductive groups. 
We also present some applications that were used in the proof of Proposition 
\ref{strange} (see Lemmas E6, E7 and E8).

Let  $\Sb$ be a  torus defined over  $\FM$. The map $\FM^\times\otimes_\ZM
Y(\Sb)\to  \Sb$ given by $x\otimes\lambda\mapsto \lambda(x)$  is an isomorphism, where
we identify $\Sb$ to the group of its points over $\FM$. Thus, if we choose
an  isomorphism  $\FM^\times\simeq  (\QM/\ZM)_{p'}$,  we get an isomorphism
$(\QM/\ZM)_{p'}\otimes_\ZM  Y(\Sb)\to  \Sb$.  Thus,  if  $\dim  \Sb=r$,  an
element  of $\Sb$ can be represented by  an element of $(\QM/\ZM)^r$ as soon
as we choose a basis of $Y(\Sb)$.

If  $\Sb$ is  a subtorus  of $\Tb$,  then the  inclusion $\Sb\subset\Tb$ is
determined by giving a basis of the sublattice $Y(\Sb)$ inside $Y(\Tb)$.
These are the basic ideas used to represent semi-simple elements in \CHEVIE.

\bigskip

\subsection{Representing reductive groups} 

A  reductive group $\Gb$ over $\FM$ is  determined up to isomorphism by the
{\em  root  datum}  $(X(\Tb),\Phi,  Y(\Tb),\Phi^\vee)$  where  $\Phi\subset
X(\Tb)$  are  the  roots  with  respect  to  the  maximal  torus  $\Tb$ and
$\Phi^\vee\subset  Y(\Tb)$ are  the corresponding  coroots. This determines
the Weyl group, a finite reflection group $W\subset\mathrm{GL}(Y(\Tb))$.

In \CHEVIE, to specify $\Gb$, we give an integral matrix $R$ whose lines
represent the simple roots in terms of a basis of $X(\Tb)$, and an integral
matrix  $R^\vee$ whose  lines represent  the simple  coroots in  terms of a
basis  of $Y(\Tb)$. It is  assumed that the bases  of $X(\Tb)$ and $Y(\Tb)$
are  chosen  such  that  the  canonical  pairing  is  given  by $\langle x,
y\rangle_\Tb=\sum_i x_i y_i$.

For convenience, two particular cases are implemented in \CHEVIE\ where the
user  just has to specify the Coxeter type  of the Weyl group $W$. If $\Gb$
is  adjoint then  $R$ is  the identity  matrix and  $R^\vee$ is  the Cartan
matrix  of the root system given by $\{\alpha^\vee(\beta)\}_{\alpha,\beta}$
where  $\alpha^\vee$ (resp.  $\beta$) runs  over the  simple coroots (resp.
simple  roots). If $\Gb$  is semi-simple simply  connected, then $\Gb^*$ is
adjoint thus the situation is reversed: $R^\vee$ is the identity matrix and
$R$  the Cartan  matrix. In  all cases,  the function  we use  constructs a
particular  integral representation  of a  Coxeter group,  so it  is called
\verb+CoxeterGroup+.

By default, the adjoint group is returned.
To illustrate this, the group $\mathrm{PGL}_3$ is obtained by

\begin{verbatim}
gap> PGL:=CoxeterGroup("A",2);
CoxeterGroup("A",2)
gap> PGL.simpleRoots;
[ [ 1, 0 ], [ 0, 1 ] ]
gap> PGL.simpleCoroots;
[ [ 2, -1 ], [ -1, 2 ] ]
\end{verbatim}

To get the semi-simple simply connected group, the additional parameter
\verb+"sc"+ has to be given. For instance, $\mathrm{SL}_3$ is obtained by

\begin{verbatim}
gap> SL:=CoxeterGroup("A",2,"sc");
CoxeterGroup("A",2,"sc")
gap> SL.simpleRoots;              
[ [ 2, -1 ], [ -1, 2 ] ]
gap> SL.simpleCoroots;            
[ [ 1, 0 ], [ 0, 1 ] ]
\end{verbatim}

To get $\mathrm{GL}_3$ we must use the general form by giving $R$ and $R^\vee$:

\begin{verbatim}
gap> GL := CoxeterGroup( [ [ -1, 1, 0], [ 0, -1, 1 ] ],
> [ [ -1, 1, 0], [ 0, -1, 1 ] ] );
CoxeterGroup([ [ -1, 1, 0 ], [ 0, -1, 1 ] ],[ [ -1, 1, 0 ], [ 0, -1, 1 ] ])
gap> GL.simpleRoots;                                   
[ [ -1, 1, 0 ], [ 0, -1, 1 ] ]
gap> GL.simpleCoroots;
[ [ -1, 1, 0 ], [ 0, -1, 1 ] ]
\end{verbatim}

More features of \CHEVIE\ will be illustrated when describing the computations
with semi-simple elements below.
\bigskip

\subsection{Some application} 
Recall  that, in order to prove Lemmas E6, E7 and E8, we had introduced the
following property:

\medskip

\begin{itemize}\itemindent0.8cm
\itemth{\SC_{\Gb_0,\Mb_0,Z,n}} {\it If $s$ is a semisimple element of $(\Mb_0/Z)^*$ which is 
quasi-isolated in $(\Mb_0/Z)^*$ and in $(\Gb_0/Z)^*$, there exists an element 
$z \in \Zb((\Mb_0/Z)^*)^\circ$ 
of order dividing $n$ such that $s$ and $sz$ are not conjugate in $(\Gb_0/Z)^*$.}
\end{itemize}

\medskip

The aim of this subsection is to show how to use the \CHEVIE~package to check the 
following lemma:

\bigskip

\begin{lem}\label{calcul}
Assume that one of the following holds:
\begin{itemize}
\itemth{1} $\Gb_0$ is of type $\Erm_6$, $\Mb_0$ is of type $\Arm_2 \times \Arm_2$, 
$Z=1$ and $n=3$.

\itemth{2} $\Gb_0$ is of type $\Erm_7$, $\Mb_0$ is of type $\Erm_6$, $Z=\Zb(\Gb_0)$ and 
$n=4$.

\itemth{3} $\Gb_0$ is of type $\Erm_7$, $\Mb_0$ is of type $\Arm_1 \times \Arm_1 \times \Arm_1$ 
as in diagram $(\Erm_7[\Arm_1^3]^\#)$, $Z=1$ and $n \in \{4,6\}$.

\itemth{4} $\Gb_0$ is of type $\Erm_8$, $\Mb_0$ is of type $\Erm_7$, $Z=1$ and $n \in \{3,5\}$.

\itemth{5} $\Gb_0$ is of type $\Erm_8$, $\Mb_0$ is of type $\Erm_6$, $Z=1$ and $n=2$.
\end{itemize}
Then $\SC_{\Gb_0,\Mb_0,Z,n}$ holds. 
\end{lem}

\begin{proof}
Note that in cases (1), (3), (4), (5), $(\Gb_0/Z)^*$ is the adjoint group of
type $E_n$, while in case (2) it is the semi-simple simply connected group
of type $E_7$.

In what follows, to simplify notations, we will set $\Gb=(\Gb_0/Z)^*$ and
$\Mb=(\Mb_0/Z)^*$. We show the complete computation corresponding to case
(1). The other cases are treated by completely similar code, excepted that
in case (2), the group should be defined via
\verb+G:=CoxeterGroup("E",7,"sc");+.

We want to show that for any $s$ which is quasi-isolated
in $\Mb$ of type $A_2\times A_2$ and $\Gb$ of type $E_6$, there is an element
of order $3$ of $\Zb(\Mb)$ which is not conjugate to $s$. 

We  first compute the list  of elements of order  3 of $\Zb(\Mb)$. 
The first thing is to specify $\Mb$.
\begin{verbatim}
gap> G:=CoxeterGroup("E",6);;PrintDiagram(G);
E6      2
        |
1 - 3 - 4 - 5 - 6
gap> M:=ReflectionSubgroup(G,[1,3,5,6]);
ReflectionSubgroup(CoxeterGroup("E",6), [ 1, 3, 5, 6 ])
\end{verbatim}
In \GAP,
the result of a command which ends with a double semicolon is not printed.
The command \verb+PrintDiagram+ shows the numbering of the simple roots.

We now compute the torus $Z(\Mb)^\circ=Z(\Mb)$.
\begin{verbatim}
gap> ZM:=AlgebraicCentre(M).Z0;
[ [ 0, 1, 0, -1, 0, 0 ], [ 0, 0, 0, 1, 0, 0 ] ]
\end{verbatim}
The torus $Z(\Mb)^\circ$ is represented by giving a basis of $Y(Z(\Mb)^\circ)$ inside
$Y(\Tb)$.

We now ask for the subgroup of elements of order 3 of $Z(\Mb)^\circ$.
\begin{verbatim}
gap> Z3:=SemisimpleSubgroup(G,ZM,3);
Group( <0,1/3,0,2/3,0,0>, <0,0,0,1/3,0,0> )
\end{verbatim}

This  group is represented as a subgroup of $\Tb$; elements of $\Tb$, which
is  of dimension 6, are represented as lists of 6 elements of $\QM/\ZM$ in
angle  brackets;  elements  of  $\QM/\ZM$  are  themselves  represented  as
fractions $r$ such that $0\le r <1$. The subgroup of elements of order 3 of
$Z(\Mb)^\circ$ is generated by 2 elements which are given above. We may ask for
the list of all elements of this group.

\begin{verbatim}
gap> Z3:=Elements(Z3);
[ <0,0,0,0,0,0>, <0,0,0,1/3,0,0>, <0,0,0,2/3,0,0>, <0,1/3,0,2/3,0,0>, 
  <0,1/3,0,0,0,0>, <0,1/3,0,1/3,0,0>, <0,2/3,0,1/3,0,0>, <0,2/3,0,2/3,0,0>, 
  <0,2/3,0,0,0,0> ]
\end{verbatim}

We now compute the list of elements quasi-isolated in both $\Gb$ and $\Mb$.
\begin{verbatim}
gap> reps:=QuasiIsolatedRepresentatives(G);
[ <0,0,0,0,0,0>, <0,0,0,0,1/2,0>, <0,0,0,1/3,0,0>, <0,1/6,1/6,0,1/6,0>, 
  <1/3,0,0,0,0,1/3> ]
\end{verbatim}
The  list  \verb+reps+  now  contains  representatives of $\Gb$-orbits of
quasi-isolated elements. The algorithm to get these was described in
\cite{bonnafe quasi}. To get all the
quasi-isolated elements in $\Tb$, we need to take the orbits under the Weyl
group:
\begin{verbatim}
gap> qi:=List(reps,s->Orbit(G,s));;
gap> List(qi,Length);
[ 1, 36, 80, 1080, 90 ]
\end{verbatim}
We have not displayed the orbits since they are quite large: the first orbit
is that of the identity element, which is trivial, but the fourth contains
$1080$ elements. We now filter
each orbit by the condition to be quasi-isolated also in $\Mb$.
\begin{verbatim}
gap> qi:=List(qi,x->Filtered(x,y->IsQuasiIsolated(M,y)));;
gap> List(qi,Length);
[ 1, 3, 26, 36, 12 ]
gap> qi[2];
[ <0,0,0,1/2,0,0>, <0,1/2,0,1/2,0,0>, <0,1/2,0,0,0,0> ]
\end{verbatim}
There  is a way to  do the same computation  which does not need to compute
the large intermediate orbits under the Weyl group of $\Gb$. The idea is to
compute  first the orbit of a semi-simple quasi-isolated representative $s$
under  representatives  of  the  double cosets $C_\Gb(s)\backslash\Gb/\Mb$,
which  are not too many,  then test for being  quasi-isolated in $\Mb$, and
finally  take the orbits  under the Weyl  group of $\Mb$.  So starting with
\verb+reps+ as above, we first compute:

\begin{verbatim}
  ce:=List(reps,s->SemisimpleCentralizer(G,s));;ce[5];
  Extended(ReflectionSubgroup(CoxeterGroup("E",6), [ 2, 3, 4, 5 ]),<(2,5,3)>)
  gap> ce[5].group;
  ReflectionSubgroup(CoxeterGroup("E",6), [ 2, 3, 4, 5 ])
  gap> ce[5].permauts;
  Group( ( 1,72, 6)( 2, 5, 3)( 7,71,11)( 8,10, 9)(12,70,16)(13,14,15)(17,68,21)
  (18,69,20)(22,66,25)(23,67,65)(26,63,28)(27,64,62)(29,59,31)(30,61,58)
  (32,57,53)(33,56,54)(34,52,48)(35,47,43)(36,42,37)(38,41,39)(44,46,45)
  (49,50,51) )
\end{verbatim}

The  first command above computes the groups $C_\Gb(s)$, which are possibly
disconnected  groups. We show for the 5th element of \verb+reps+ how such a
group is represented: it contains a reflection subgroup
of  the Weyl group of $\Gb$, the Weyl group of $C_\Gb^\circ(s)$, obtained above
as  \verb+ce[5].group+,  extended  by  the  group  of diagram automorphisms
induced on it by $C_\Gb(s)$, obtained above as \verb+ce[5].permauts+; these
automorphisms  are  denoted  by  the  permutation  of  the  simple roots of
$C_\Gb^\circ(s)$ they induce.

To get the whole Weyl group of $C_\Gb(s)$ we need to combine these two pieces.
For this we define a \GAP\ function:

\begin{verbatim}
TotalGroup:=g->Subgroup(G,Concatenation(g.group.generators,g.permauts.generators));
\end{verbatim}

We then compute representatives of the double cosets 
$C_\Gb(s)\backslash\Gb/\Mb$, we apply them to \verb+reps+, keep
the ones still quasi-simple in $\Mb$:

\begin{verbatim}
dreps:=List(ce,g->List(DoubleCosets(G,TotalGroup(g),M),Representative));;
qi:=List([1..Length(reps)],i->List(dreps[i],w->reps[i]^w));;
qi:=List(qi,x->Filtered(x,y->IsQuasiIsolated(M,y)));
[ [ <0,0,0,0,0,0> ], [ <0,1/2,0,0,0,0>, <0,0,0,1/2,0,0>, <0,1/2,0,1/2,0,0> ], 
  [ <0,0,0,1/3,0,0>, <0,0,0,2/3,0,0>, <1/3,2/3,1/3,0,2/3,2/3>, 
      <1/3,1/3,1/3,0,2/3,2/3> ], 
  [ <1/3,1/2,1/3,1/2,2/3,2/3>, <1/3,1/2,1/3,0,2/3,2/3>, 
      <2/3,0,2/3,5/6,2/3,2/3> ], [ <1/3,0,1/3,0,1/3,1/3> ] ]
\end{verbatim}

We get a list such that the $\Mb$-orbits of the sublists give the same list
as  before. We will  need this previous  list of all $\Gb$-conjugates which
are  $\Mb$-quasi-isolated, so if we did not  keep it we recompute this list
containing the $\Mb$-orbits of the sublists by

\begin{verbatim}
qim:=List(qi,l->Union(List(l,s->Orbit(M,s))));;
\end{verbatim}

We  now ask, for each element $s$ of each our orbits, how many elements $z$
of \verb+Z3+ are such that $s$ and $sz$ are not $\Gb$-conjugate. The test
for being conjugate is that $sz$ is in the same $\Gb$-orbit. We need to make
the test only for our representatives of the $\Mb$-orbits, since if
$s$ is $\Gb$-conjugate to $sz$ with $z\in\Zb(\Mb)$, then 
$m s m^{-1}$ is $\Gb$-conjugate to $m s m^{-1}z=m sz m^{-1}$.

\begin{verbatim}
gap> List([1..Length(qi)],i->List(qi[i],s->Number(Z3,
  z->PositionProperty(qim,o->s*z in o)<>i)));
[ [ 8 ], [ 8, 8, 8 ], [ 7, 7, 3, 3 ], [ 6, 6, 6 ], [ 6 ] ]
\end{verbatim}
and we find indeed that there is always more than $0$ elements $z$ which work.
Note that the function \verb+PositionProperty+ returns \verb+false+ when no
element is found satisfying the given property, thus the number counted is
the $z$ such that $s$ and $sz$ are in a different orbit, as well as the cases
when $sz$ is not quasi-isolated in $\Gb$.
\end{proof}

\bigskip

\subsection{Rational structures}
We now show the \CHEVIE\ code for the following lemma. Here again to simplify
notations we note $\Gb$ for $(\Gb_0/Z)^*$ and $\Mb$ for $(\Mb_0)^*$. We are
going to show the \CHEVIE\ code to prove the following lemma:

\begin{lem}\label{ordre 8}
If  $\Gb$ is adjoint of  type $\Erm_7$, if $\Mb$  is of type $\Arm_1 \times
\Arm_1  \times \Arm_1$ as in diagram $(\Erm_7[\Arm_1^3]^\#)$, 
if $q  \in \{3,5\}$ and if $\chi_{\Zb(\Mb),F}=\Phi_1^a\Phi_2^b$ with $a$, $b \ge 1$, 
then $\Zb(\Mb)^F$ contains an element of order $8$.
\end{lem}

\bigskip

In  \CHEVIE, to specify an $\gfq$-structure on a reductive group, we must in
addition  give an element $\phi\in\mathrm{GL}(Y(\Tb))$ such that $F=q\phi$.
We  may chose $\phi$ such that it  stabilizes the set of simple roots. Such
an element $\phi$ is determined by the coset
$W\phi\subset\mathrm{GL}(Y(\Tb))$, so the structure which represents it in
\CHEVIE\ is called a {\em Coxeter coset}.

Further,  if $\Mb'$ is  an $F$-stable $\Gb$-conjugate  of the Levi subgroup
$\Mb$,  the pair $(\Mb',F)$  is isomorphic to  $(\Mb,wF)$ for some $w\in W$
(determined  by $\Mb'$ up  to  $F$-conjugacy).  So,  given  a Coxeter coset
$W\phi$,  an $F$-stable conjugate of a Levi  subgroup whose Weyl group is a
standard  parabolic subgroup $W_I$ is represented by a subcoset of the form
$W_I w\phi$, where $w\phi$ normalizes $W_I$.

To  check the lemma,  we first compute  the list of  elements of order 8 of
$\Zb(\Mb)$, using the same commands as shown before.
\begin{verbatim}
gap> G:=CoxeterGroup("E",7);;PrintDiagram(G);
E7      2
        |
1 - 3 - 4 - 5 - 6 - 7
gap> M:=ReflectionSubgroup(G,[2,5,7]);;
gap> ZM:=AlgebraicCentre(M);;
gap> Z8:=SemisimpleSubgroup(G,ZM.Z0,8);
Group( <1/8,0,0,0,0,0,0>, <0,0,1/8,7/8,0,1/8,0>, <0,0,0,1/8,0,7/8,0>,
<0,0,0,0,0,1/8,0> )
gap> Z8:=Elements(Z8);;Length(Z8);
4096
\end{verbatim}

We  now ask for representatives of the $\Gb^F$-classes of $F$-stable
$\Gb$-conjugates   of   $\Mb$. The group $\Gb$ is split, so $\phi$ is trivial.
Thus an $F$-stable-conjugate of $\Mb$ is represented by a coset of the
form $W_I w$. We first ask for  the list of all possible such
twistings of $\Mb$:
\begin{verbatim}
gap> Mtwists:=Twistings(G,M);
[ A1<2>xA1<5>xA1<7>.(q-1)^4, (A1xA1xA1)<2,5,7>.(q-1)^2*(q^2+q+1), 
  A1<2>xA1<5>xA1<7>.(q-1)^2*(q^2+q+1), (A1xA1xA1)<2,5,7>.(q^2+q+1)^2, 
  (A1xA1xA1)<2,7,5>.(q-1)*(q+1)*(q^2+q+1), 
  (A1xA1xA1)<2,7,5>.(q-1)*(q+1)*(q^2-q+1), (A1xA1xA1)<2,7,5>.(q+1)^2*(q^2-q+1)
    , A1<2>xA1<5>xA1<7>.(q+1)^2*(q^2-q+1), 
  (A1xA1)<2,7>xA1<5>.(q-1)*(q+1)*(q^2+q+1), 
  (A1xA1)<2,7>xA1<5>.(q-1)*(q+1)*(q^2-q+1), (A1xA1xA1)<2,7,5>.(q^2-q+1)^2, 
  (A1xA1xA1)<2,5,7>.(q^4-q^2+1), A1<2>xA1<5>xA1<7>.(q+1)^4, 
  A1<2>xA1<5>xA1<7>.(q^2+1)^2, (A1xA1)<2,7>xA1<5>.(q^4+1), 
  A1<2>xA1<5>xA1<7>.(q-1)^2*(q+1)^2, (A1xA1)<2,7>xA1<5>.(q+1)^2*(q^2+1), 
  (A1xA1)<2,7>xA1<5>.(q-1)^2*(q^2+1), A1<2>xA1<5>xA1<7>.(q-1)*(q+1)*(q^2+1), 
  (A1xA1)<2,7>xA1<5>.(q-1)*(q+1)*(q^2+1), (A1xA1)<2,7>xA1<5>.(q-1)^3*(q+1), 
  A1<2>xA1<5>xA1<7>.(q-1)*(q+1)^3, A1<2>xA1<5>xA1<7>.(q-1)^3*(q+1), 
  (A1xA1)<2,7>xA1<5>.(q-1)*(q+1)^3, (A1xA1)<2,7>xA1<5>.(q-1)^2*(q+1)^2 ]
\end{verbatim}
In the above list, brackets around pairs or triples of $A_1$ denote an
orbit of the Frobenius on the components. The element $w$ is not displayed,
but the order of $|\Zb(\Mb)^{wF}|$ is displayed. We want to keep
the sublist where that order is a product of $\Phi_1(q)$ and $\Phi_2(q)$.

\begin{verbatim}
gap> Mtwists:=Filtered(Mtwists,MF->Set(PhiFactors(MF))=[-1,1]);
[ A1<2>xA1<5>xA1<7>.(q+1)^4, A1<2>xA1<5>xA1<7>.(q-1)^2*(q+1)^2, 
  (A1xA1)<2,7>xA1<5>.(q-1)^3*(q+1), A1<2>xA1<5>xA1<7>.(q-1)*(q+1)^3, 
  A1<2>xA1<5>xA1<7>.(q-1)^3*(q+1), (A1xA1)<2,7>xA1<5>.(q-1)*(q+1)^3, 
  (A1xA1)<2,7>xA1<5>.(q-1)^2*(q+1)^2 ]
\end{verbatim}
Here  \verb+PhiFactors+ gives the  eigenvalues of $w$  on the invariants of
the Weyl group of $\Mb$ acting on the symmetric algebra of
$X(\Tb)\otimes\CM$.  The cases  we want  is when  these eigenvalues are all
equal  to  $1$  or  $-1$  (actually  this  gives  use one extra case, where
$|\Zb(\Mb)^{wF}|=(q+1)^4$  since  the  eigenvalues  on  the  complement  of
$\Zb(\Mb)$  are always $1$; we will just  have to disregard the first entry
of \verb+Mtwists+).

Now  for each of the remaining \verb+Mtwists+ we compute the fixed points of
$wF$ on \verb+Z8+, and look at the maximal order of an element in there. We
first  illustrate the necessary commands one by one on an example before 
showing a line of code which combines them.
\begin{verbatim}
gap> Z8F:=Filtered(Z8,s->Frobenius(Mtwists[3])(s)^3=s);
[ <0,0,0,0,0,0,0>, <0,0,0,1/4,0,1/4,0>, <0,0,0,1/2,0,1/2,0>, 
  <0,0,0,3/4,0,3/4,0>, <0,0,1/2,0,0,1/2,0>, <0,0,1/2,1/4,0,3/4,0>, 
  <0,0,1/2,1/2,0,0,0>, <0,0,1/2,3/4,0,1/4,0>, <1/4,0,1/4,1/8,0,7/8,0>, 
  <1/4,0,1/4,3/8,0,1/8,0>, <1/4,0,1/4,5/8,0,3/8,0>, <1/4,0,1/4,7/8,0,5/8,0>, 
  <1/4,0,3/4,1/8,0,3/8,0>, <1/4,0,3/4,3/8,0,5/8,0>, <1/4,0,3/4,5/8,0,7/8,0>, 
  <1/4,0,3/4,7/8,0,1/8,0>, <1/2,0,0,0,0,0,0>, <1/2,0,0,1/4,0,1/4,0>, 
  <1/2,0,0,1/2,0,1/2,0>, <1/2,0,0,3/4,0,3/4,0>, <1/2,0,1/2,0,0,1/2,0>, 
  <1/2,0,1/2,1/4,0,3/4,0>, <1/2,0,1/2,1/2,0,0,0>, <1/2,0,1/2,3/4,0,1/4,0>, 
  <3/4,0,1/4,1/8,0,7/8,0>, <3/4,0,1/4,3/8,0,1/8,0>, <3/4,0,1/4,5/8,0,3/8,0>, 
  <3/4,0,1/4,7/8,0,5/8,0>, <3/4,0,3/4,1/8,0,3/8,0>, <3/4,0,3/4,3/8,0,5/8,0>, 
  <3/4,0,3/4,5/8,0,7/8,0>, <3/4,0,3/4,7/8,0,1/8,0> ]
\end{verbatim}
The   expression  \verb+Frobenius(Mtwists[3])+  returns  a  function  which
applies the $w$ for the 3rd twisting of $\Mb$, described as
\verb|(A1xA1)<2,7>xA1<5>.(q-1)^3*(q+1)|,  to its argument.  To compute $wF$
we  still have to raise to the third power since $q=3$. We can see from the
denominators  that  some  elements  in  the  resulting  list of $wF$-stable
elements  of \verb+Z8+ are  of order 8.  We can make  this easier to see by
writing a small function:
\begin{verbatim}
gap> OrderSemisimple:=s->Lcm(List(s.v,Denominator));
gap> List(Z8F,OrderSemisimple);
[ 1, 4, 2, 4, 2, 4, 2, 4, 8, 8, 8, 8, 8, 8, 8, 8, 2, 4, 2, 4, 2, 4, 2, 4, 8, 
  8, 8, 8, 8, 8, 8, 8 ]
gap> Set(last);
[ 1, 2, 4, 8 ]
\end{verbatim}
We now do the computation simultaneously for all cosets:
\begin{verbatim}
gap> List(Mtwists,MF->Set(List(Filtered(Z8,s->Frobenius(MF)(s)^3=s),
> OrderSemisimple)));
[ [ 1, 2, 4 ], [ 1, 2, 4, 8 ], [ 1, 2, 4, 8 ], [ 1, 2, 4, 8 ], 
  [ 1, 2, 4, 8 ], [ 1, 2, 4, 8 ], [ 1, 2, 4, 8 ] ]
\end{verbatim}
and we see that indeed, apart from the first twist which should be disregarded,
for all twists the fixed points of \verb+Z8+ still contain elements of order 8.

\end{document}